\documentclass[12pt,reqno]{amsart}

\usepackage{amssymb}
\usepackage{mathtools}

\usepackage{hyperref}

\newtheorem{theorem}{Theorem}[section]
\newtheorem{proposition}[theorem]{Proposition}
\newtheorem{lemma}[theorem]{Lemma}
\newtheorem{corollary}[theorem]{Corollary}
\newtheorem{conjecture}[theorem]{Conjecture}

\theoremstyle{definition}
\newtheorem{definition}[theorem]{Definition}

\numberwithin{equation}{section}

\begin{document}

\title[Vector bundles on Riemann surface with projective structure]{Vector bundles and 
connections on Riemann surfaces with projective structure}

\author[I. Biswas]{Indranil Biswas}

\address{School of Mathematics, Tata Institute of Fundamental Research,
Homi Bhabha Road, Mumbai 400005}

\email{indranil@math.tifr.res.in}

\author[J. Hurtubise]{Jacques Hurtubise}

\address{Department of Mathematics, McGill University, Burnside
Hall, 805 Sherbrooke St. W., Montreal, Que. H3A 2K6, Canada}

\email{jacques.hurtubise@mcgill.ca}

\author[V. Roubtsov]{Vladimir Roubtsov}

\address{UNAM, LAREMA UMR 6093 du CNRS, Universit\'e d'Angers, 49045 Cedex 01,
Angers, France}

\email{vladimir.roubtsov@univ-angers.fr}

\subjclass[2010]{14H60, 16S32, 14D21, 53D30}

\keywords{Projective structure, differential operator, holomorphic connection,
oper, torsor}

\date{}

\begin{abstract}
Let ${\mathcal B}_g(r)$ be the moduli space of triples of the form $(X,\, K^{1/2}_X,\, F)$, where $X$ is 
a compact connected Riemann surface of genus $g$, with $g\, \geq\, 2$, $K^{1/2}_X$ is a theta 
characteristic on $X$, and $F$ is a stable vector bundle on $X$ of rank $r$ and degree zero. We 
construct a $T^*{\mathcal B}_g(r)$--torsor ${\mathcal H}_g(r)$ over ${\mathcal B}_g(r)$. This 
generalizes on the one hand the torsor over the moduli space of stable vector bundles of rank $r$, on a 
fixed Riemann surface $Y$, given by the moduli space of holomorphic connections on the
stable vector bundles of rank $r$ on $Y$, and 
on the other hand the torsor over the moduli space of Riemann surfaces given by the moduli space of 
Riemann surfaces with a projective structure. It is shown that ${\mathcal H}_g(r)$ has a holomorphic 
symplectic structure compatible with the $T^*{\mathcal B}_g(r)$--torsor structure. We also describe 
${\mathcal H}_g(r)$ in terms of the second order matrix valued differential operators. It is shown that 
${\mathcal H}_g(r)$ is identified with the $T^*{\mathcal B}_g(r)$--torsor given by the sheaf of 
holomorphic connections on the theta line bundle over ${\mathcal B}_g(r)$.
\end{abstract}

\maketitle

\tableofcontents

\section{Introduction}\label{sec1}

Let $X$ be a compact connected Riemann surface of genus $g$, with $g\, \geq \,2$, and ${\mathcal 
N}_X(r)$ the moduli space of stable vector bundles on $X$ of rank $r$ and degree zero. Let ${\mathcal 
C}_X(r)$ be the moduli space of rank $r$ holomorphic connections $(E,\, D)$ on $X$ such that the 
underlying holomorphic vector bundle $E$ is stable. It has a forgetful map to ${\mathcal N}_X(r)$ that 
sends any $(E,\, D)$ to $E$, and ${\mathcal C}_X(r)$ is a torsor over ${\mathcal N}_X(r)$ for the 
holomorphic cotangent bundle $T^*{\mathcal N}_X(r)$. Moreover, ${\mathcal C}_X(r)$ has a holomorphic 
symplectic structure which is compatible with the $T^*{\mathcal N}_X(r)$--torsor structure.

There is another $T^*{\mathcal N}_X(r)$--torsor that one can define. To each bundle $E\, \in\, 
{\mathcal N}_X(r)$ we associate the Quillen determinant line for its $\overline\partial$ operator, and 
this defines a holomorphic line bundle $L$ over ${\mathcal N}_X(r)$. Then consider the sheaf of 
pointwise connections ${\rm Conn}(L)$ over ${\mathcal N}_X(r)$ (i.e., the sheaf on ${\mathcal N}_X(r)$ 
whose sections are holomorphic connections on $L$). This is again a $T^*{\mathcal N}_X(r)$--torsor
over ${\mathcal N}_X(r)$.

The surprise, established in \cite{BH1, BH2} is that there is a canonical isomorphism between these 
two torsors. The isomorphism is constructed by taking $C^\infty$ sections on both sides that at the first 
sight do not seem to have much to do with each other, but have the same data encoding the torsor. For 
${\mathcal C}_X(r)$, we associate to each bundle its unitary connection given by the 
Narasimhan--Seshadri theorem; for ${\rm Conn}(L)$ the Chern connection for the Quillen metric on $L$. There is 
a holomorphic version of this equivalence by sections, expanded in \cite{BH2}, in which the line bundle 
$L$ is restricted to the complement $\mathcal U$ of the theta divisor in ${\mathcal N}_X(r)$ associated 
to a theta characteristic $K^{1/2}_X$ of $X$, so the canonical trivialization of $L$ over $\mathcal U$ 
produces an integrable holomorphic connection on $L\big\vert_{\mathcal U}$. On the other hand, a 
holomorphic connection on any bundle $E\, \in\, \mathcal U$ is obtained as the asymptotic data of a canonical 
section of $(E\otimes K^{1/2}_X)\boxtimes (E^*\otimes K^{1/2}_X)$ over $X\times X$, with singularities 
on the diagonal. There is a holomorphic isomorphism between ${\rm Conn}(L)$ and ${\mathcal C}_X(r)$ that 
maps the section of ${\rm Conn}(L)\big\vert_{\mathcal U}$ given be the trivialization of 
$L\big\vert_{\mathcal U}$ to the section of ${\mathcal C}_X(r)\big\vert_{\mathcal U}$ given by the above 
canonical section. In fact, not only is there an isomorphism, but it is obtained by linking together 
rather special and hitherto unrelated objects on each side of the equivalence, in several different 
ways.

The holomorphic version allows us, of course to move not only the bundle, but the base curve as well, and the 
question we are examining in this paper is whether we can again have two a priori inequivalent 
holomorphic $T^*{\mathcal B}_g(r)$ torsors over the space ${\mathcal B}_g(r)$ of pairs (bundles, curves) 
actually turn out to be the same. The answer turns out to be yes. As a bonus, we have a hereto undefined 
symplectic structure for one of the torsors.

To explain with more details, let ${\mathcal P}_g$ denote the moduli space of Riemann surfaces of genus 
$g$ equipped with a projective structure. It has a natural map to the moduli space ${\mathcal M}_g$ of 
Riemann the surfaces of genus $g$ that simply forgets the projective structure, and ${\mathcal P}_g$ is 
a holomorphic torsor over ${\mathcal M}_g$ for the holomorphic cotangent bundle $T^*{\mathcal M}_g$. 
Also, ${\mathcal P}_g$ has a holomorphic symplectic structure which is compatible with the $T^*{\mathcal 
M}_g$--torsor structure.

To combine the torsors ${\mathcal P}_g$ and ${\mathcal C}_X(r)$ into a single one, let ${\mathcal B}_g(r)$
be the moduli space of triples of the form $(X,\, K^{1/2}_X,\, F)$, where
\begin{itemize}
\item $X$ is a compact connected Riemann surface of genus $g$, with $g\, \geq\, 2$,

\item $K^{1/2}_X$ is a theta characteristic on $X$ (the holomorphic cotangent bundle of $X$ is denoted 
by $K_X$), and

\item $F$ is a stable vector bundle over $X$ of rank $r$ and degree zero.
\end{itemize}

Fix a theta characteristic $K^{1/2}_X$ on a compact Riemann surface $X$ of genus $g$. Given 
a stable vector bundle $F$ on $X$ of rank $r$ and degree zero, we construct a certain 
quotient space of the space of all holomorphic connections on the first order jet bundle 
$J^1(F\otimes (K^{1/2}_X)^*)$; this quotient space is denoted by ${\mathcal D}(F)$ (see
\eqref{qe} and Corollary \ref{cor1}). Let ${\mathcal H}_g(r)$ denote the moduli space of
objects of the form $(X,\, K^{1/2}_X,\, F,\, D)$, where $(X,\, K^{1/2}_X,\, F)\, \in\, 
{\mathcal B}_g(r)$ and $D\, \in\, {\mathcal D}(F)$. It has a natural map
$$
\gamma\, :\, {\mathcal H}_g(r)\, \longrightarrow\, {\mathcal B}_g(r)
\, ,\ \ \ (X,\, K^{1/2}_X,\, F,\, D)\, \longmapsto\,(X,\, K^{1/2}_X,\, F)\, .
$$

We prove the following (see Theorem \ref{thm1} and Corollary \ref{cor5}):

\begin{theorem}\label{thm-i1}\mbox{}
\begin{enumerate}
\item ${\mathcal H}_g(r)$ is a torsor over ${\mathcal B}_g(r)$ for
the holomorphic cotangent bundle $T^*{\mathcal B}_g(r)$.

\item ${\mathcal H}_g(r)$ has an algebraic symplectic structure $\Omega_{{\mathcal H}_g(r)}$.

\item The symplectic form $\Omega_{{\mathcal H}_g(r)}$ on ${\mathcal H}_g(r)$ is compatible with the
$T^*{\mathcal B}_g(r)$--torsor structure of ${\mathcal H}_g(r)$.

\item There is a holomorphic line bundle $\mathbf{L}$ on ${\mathcal H}_g(r)$, and a holomorphic
connection $\nabla^{\mathbf{L}}$ on $\mathbf{L}$, such that the curvature of $\nabla^{\mathbf{L}}$
is the symplectic form $\Omega_{{\mathcal H}_g(r)}$.
\end{enumerate}
\end{theorem}

Consider the standard subbundle $F\otimes (K^{1/2}_X)^*\otimes K_X\,=\, F\otimes K^{1/2}_X\, 
\subset\, J^1(F\otimes (K^{1/2}_X)^*)$. For any holomorphic connection $\mathbb D$ on 
$J^1(F\otimes (K^{1/2}_X)^*)$, the second fundamental form of $F\otimes K^{1/2}_X$ for 
$\mathbb D$ is the identity map of $F$ (see Corollary \ref{cor1}). This property is very 
similar to the defining property of opers. We recall that opers were introduced by Beilinson 
and Drinfeld \cite{BD}, \cite{BD2}. Their motivation came from the works of Drinfeld and Sokolov 
\cite{DS1}, \cite{DS2}. We note that the study of opers within geometry and mathematical 
physics has received much attention in the recent times.

There is a natural divisor $\Theta\,\subset\, {\mathcal B}_g(r)$ consisting of all
$(X,\, K^{1/2}_X,\, F)$ such that $H^0(X,\, F\otimes K^{1/2}_X)\, \not=\, 0$. The
line bundle on ${\mathcal B}_g(r)$ defined by $\Theta$ will be denoted by $\mathcal L$. Let
$$
{\mathcal C}({\mathcal L})\,\longrightarrow\, {\mathcal B}_g(r)
$$
be the holomorphic fiber bundle defined by the sheaf of
holomorphic connection on $\mathcal L$. So the 
space of holomorphic sections of ${\mathcal C}({\mathcal L})$ over an open subset $U\, 
\subset\, {\mathcal B}_g(r)$ is the space of all holomorphic connections on ${\mathcal 
L}\big\vert_U\, \longrightarrow\, U$. This ${\mathcal C}({\mathcal L})$ is an algebraic torsor over 
${\mathcal B}_g(r)$ for the holomorphic cotangent bundle $T^*{\mathcal B}_g(r)$.

We prove the following (see Theorem \ref{thm3}):

\begin{theorem}\label{thm-i2}
There is a canonical algebraic isomorphism of $T^*{\mathcal B}_g(r)$--torsors
$$
{\mathcal H}\, :\, {\mathcal H}_g(r)\, \longrightarrow\, {\mathcal C}({\mathcal L})\, .
$$
\end{theorem}

Projective structures on a Riemann surface $X$ are defined by giving a holomorphic coordinate
atlas on $X$ such that all the transition functions are M\"obius transformations.
Projective structures on a Riemann surface $X$ are identified with holomorphic
ordinary differential operators $\mathcal D$ on $X$ of order two such that
\begin{itemize}
\item the symbol of $\mathcal D$ is the constant function $1$ on $X$, and

\item the sub-leading term of $\mathcal D$ vanishes identically (equivalently, the
$\text{GL}(2, {\mathbb C})$--local system of $X$ defined by the sheaf of solutions of $\mathcal D$ 
is actually a $\text{SL}(2, {\mathbb C})$--local system).
\end{itemize}

The above mentioned space ${\mathcal H}_g(r)$ admits a similar description in terms of the second 
order matrix valued differential operators. To explain this, for any
$(X,\, K^{1/2}_X,\, F)\, \in\, {\mathcal B}_g(r)$, let
$$
\widetilde{\mathcal D}'(X,\, K^{1/2}_X,\, F)\, \subset\,
H^0(X,\, \text{Diff}^2_X(F\otimes K^{-1/2}_X,\, F\otimes K^{3/2}_X))
$$
be the locus of holomorphic
differential operators whose symbol is ${\rm Id}_F\, \in\, H^0(X,\, \text{End}(F))$.
The vector space $H^0(X,\, \text{End}(F)\otimes K^{\otimes 2}_X)$ acts freely on
$\widetilde{\mathcal D}'(X,\, K^{1/2}_X,\, F)$ (but this action is not transitive).

We prove the following (see Theorem \ref{thm2}):

\begin{theorem}\label{thm-i3}
There is a canonical bijection between ${\mathcal D}(F)$ and the quotient space
$$\widetilde{\mathcal D}'(X,\, K^{1/2}_X,\, F)\big{/}H^0(X,\, {\rm ad}(F)\otimes K^{\otimes 2}_X)\, .$$
\end{theorem}

There is a natural holomorphic projection ${\mathcal H}_g(r)\, \longrightarrow\, {\mathcal P}_g$
(see Proposition \ref{prop3}), using which we
may pullback, to ${\mathcal H}_g(r)$, the symplectic $2$-form on ${\mathcal P}_g$. On the 
other hand, using isomonodromic deformations, the symplectic form on ${\mathcal C}_X(r)$ produces
a holomorphic $2$-form on ${\mathcal H}_g(r)$. It is natural to ask whether the symplectic form
on ${\mathcal H}_g(r)$ in Theorem \ref{thm-i1}(2) is a combination of these two $2$-forms (see
Conjecture \ref{conj} for a precise formulation).

The result of this paper, as well as those of \cite{BH1}, \cite{BH2}, have a flavor of geometric
quantization: a symplectic manifold projects to a manifold of half the dimension, equipped with a
line bundle with a connection. The fibers of the projection are Lagrangians, and the curvature
of the connection is the symplectic form. We are unaware, however, of any concrete link to
geometric quantization.

\section{Stable {\it r}--opers and their properties}\label{sec2}

\subsection{Stable {\it r}--opers}\label{sec2.1}

Let $X$ be a compact connected Riemann surface; its holomorphic
cotangent (respectively, tangent) bundle will be denoted by $K_X$
(respectively, $TX$). We shall assume throughout that $$\text{genus}(X)\, =\, g \, \geq\, 2\, .$$

A {\it holomorphic
connection} on a holomorphic vector bundle $E$ over $X$ is a first order
holomorphic differential operator
$$
D\, :\, E\, \longrightarrow\, E\otimes K_X
$$
such that $D(fs) \,=\, fD(s)+s\otimes df$ for all locally defined holomorphic
functions $f$ on $X$ and all locally defined holomorphic sections $s$ of $E$;
see \cite{At}. A holomorphic connection on $E$ is automatically flat because
$\Omega^{2,0}_X\,=\, 0$.

If $D$ is a holomorphic connection on $E$, and $S\, \subset\, E$ is a holomorphic
subbundle of $E$, then consider the composition of homomorphisms
$$
S\, \hookrightarrow\, E\, \stackrel{D}{\longrightarrow}\, E\otimes K_X
\, \xrightarrow{q\otimes{\rm Id}_{K_X}}\, (E/S)\otimes K_X\, ,
$$
where
\begin{equation}\label{eq}
q\, :\, E\, \longrightarrow\, E/S
\end{equation}
is the quotient map. This composition
of homomorphisms is ${\mathcal O}_X$--linear, and hence it corresponds to a holomorphic section
\begin{equation}\label{e1}
D_S\,\in\, H^0(X,\, {\rm Hom}(S,\, E/S)\otimes K_X)\, .
\end{equation}
This homomorphism $D_S$ is called the \textit{second fundamental form} of $S$ for the connection $D$.

A holomorphic vector bundle $V$ on $X$ is called stable if $$\frac{\text{degree}(W)}{\text{rank}(W)}\, <\,
\frac{\text{degree}(V)}{\text{rank}(V)}$$ for every holomorphic subbundle $0\, \not=\, W\,\subsetneq\, V$.

\begin{definition}\label{defro}
For a positive integer $r$, a $r$--\textit{oper} on $X$ is a triple $(E,\, S,\, D)$, where
\begin{itemize}
\item $E$ is a holomorphic vector bundle on $X$ of rank $2r$,

\item $D$ is a holomorphic connection on $E$, and

\item $S\, \subset\, E$ is a holomorphic subbundle of rank $r$,
\end{itemize}
such that the second fundamental form $D_S$ defined in \eqref{e1} is an isomorphism.

A $r$--oper $(E,\, S,\, D)$ is called \textit{stable} if the holomorphic
vector bundle $S$ of rank $r$ is stable.
\end{definition}

Using the isomorphism $D_S \, :\, S \, \longrightarrow\, (E/S)\otimes K_X$ in Definition
\ref{defro}, we may identify $E/S$ with $S\otimes TX$. Invoking this isomorphism $D_S$ is made into the
identity map of $S$. Henceforth, for any $r$--oper we shall always execute the option of using this
isomorphism.

Fix a holomorphic line bundle $K^{1/2}_X$ on $X$ such that $K^{1/2}_X\otimes
K^{1/2}_X\,=\, K_X$; also fix a holomorphic isomorphism of $K^{1/2}_X\otimes
K^{1/2}_X$ with $K_X$. Such a line bundle is called a \textit{theta characteristic}
on $X$. The dual line bundle $(K^{1/2}_X)^*$ will be denoted by $K^{-1/2}_X$; also
$(K^{1/2}_X)^{\otimes n}$ (respectively, $(K^{1/2}_X)^{-\otimes n}$) will be denoted by
$K^{n/2}_X$ (respectively, $K^{-n/2}_X$) for every $n\, \geq\, 1$.

Take a $r$--oper $(E,\, S,\, D)$ on $X$. Define the vector bundle $F\, :=\,S\otimes K^{-1/2}_X$,
so we have 
\begin{equation}\label{ew}
S\,=\, F\otimes K^{1/2}_X\, .
\end{equation}
Then from the isomorphism $D_S$ we have
\begin{equation}\label{ew2}
E/S\,=\, S\otimes (K_X)^*\,=\, F\otimes K^{1/2}_X\otimes (K_X)^* \,=\, F\otimes K^{-1/2}_X\, .
\end{equation}
Since $E$ admits a holomorphic connection, we have
\begin{equation}\label{e2}
\text{degree}(E)\,=\, 0
\end{equation}
\cite[p.~202, Proposition 18(i)]{At}. Using \eqref{e2}, \eqref{ew2} and \eqref{ew} it follows that
$$
0\,=\, \text{degree}(E)\,=\, \text{degree}(S)+\text{degree}(E/S)
$$
$$
=\, \text{degree}(F\otimes K^{1/2}_X)
+\text{degree}(F\otimes K^{-1/2}_X)\,=\, 2\cdot \text{degree}(F)\, .
$$
Hence we have $\text{degree}(F)\,=\, 0$.

For any holomorphic vector bundle $W$ on $X$, the first jet bundle $J^1(W)$
fits into the following short exact sequence of holomorphic vector bundles on $X$:
\begin{equation}\label{ej}
0 \, \longrightarrow\, W\otimes K_X \, \stackrel{\iota}{\longrightarrow}\, J^1(W) \,
\longrightarrow\, J^0(W) \,=\, W \, \longrightarrow\,0\,.
\end{equation}

\begin{lemma}\label{lem1}
Let $(E,\, S,\, D)$ be a $r$--oper on $X$. Then the first jet bundle $J^1(E/S)$
is canonically identified with the holomorphic vector bundle $E$. Also, $E$ is
holomorphically identified with $J^1(S\otimes TX)$.
\end{lemma}

\begin{proof}
Using the flat holomorphic connection $D$ we shall construct a homomorphism
\begin{equation}\label{e3}
\varphi\, :\, E\, \longrightarrow\, J^1(E/S)\, .
\end{equation}
For this, take any point $x\, \in\, X$ and any element
$v\, \in\, E_x$ in the fiber over $x$. Let $\widehat{v}$ be the unique flat section of $E$ (for the
connection $D$), defined
on a simply connected open neighborhood $U\, \subset\, X$ of $x$, such that $\widehat{v}(x)\,
=\,v$. Now restrict the section $q(\widehat{v})\, \in\, H^0(U,\, E/S)$, where
$q$ is the quotient map in \eqref{eq}, to the first order infinitesimal neighborhood
of $x$; let $\widehat{v}'\, \in\, J^1(E/S)_x$ be the element obtained this
way from $q(\widehat{v})$. The map $\varphi$ in \eqref{e3} sends any $v\, \in\, E_x$,
$x\, \in\, X$, to $\widehat{v}'\, \in\, J^1(E/S)_x$ constructed above from it.

The map $\varphi$ in \eqref{e3} fits in the following commutative diagram of homomorphisms
\begin{equation}\label{e4}
\begin{matrix}
0 &\longrightarrow & S &\longrightarrow & E &\stackrel{q}{\longrightarrow} & E/S &\longrightarrow & 0\\
&& \,\,\,\,\Big\downarrow\varphi' && \,\,\,\Big\downarrow\varphi && \Big\Vert\\
0 &\longrightarrow & (E/S)\otimes K_X &\longrightarrow & J^1(E/S) &\longrightarrow &
E/S &\longrightarrow & 0
\end{matrix}
\end{equation}
where $\varphi'$ is the restriction of $\varphi$ to the subbundle $S$, and the exact sequence at
the bottom of \eqref{e4} is the one in \eqref{ej} for $W\,=\, E/S$. It is straightforward to check
that the homomorphism $\varphi'$ in \eqref{e4} actually coincides with the second fundamental
form $D_S$ defined in \eqref{e1}. Since $D_S$ is an isomorphism, from the commutativity of
\eqref{e4} it follows immediately that $\varphi$ constructed in \eqref{e3} is an isomorphism.

Using the isomorphism $E/S\,\stackrel{\sim}{\longrightarrow}\, S\otimes TX$ in \eqref{ew2}, the
isomorphism $\varphi$ identifies $E$ with $J^1(S\otimes TX)$.
\end{proof}

\begin{lemma}\label{lem1b}
Let $(E,\, S,\, D)$ be a $r$--oper on $X$. Let $D^1\, :=\, (\varphi^{-1})^* D$ be the connection on
$J^1(E/S)$, where $\varphi$ is the isomorphism in \eqref{e4}. Then the second fundamental form of
$(E/S)\otimes K_X \,\hookrightarrow \, J^1(E/S)$ for the connection $D^1$ is the identity
map of $(E/S)\otimes K_X$.
\end{lemma}

\begin{proof}
This follows immediately from the construction of $\varphi$, the commutativity of the diagram
in \eqref{e4} and the fact that $\varphi'$ coincides with the second fundamental form
$D_S$ defined in \eqref{e1}.
\end{proof}

\begin{proposition}\label{prop1}
Let $F$ be a stable vector bundle on $X$ of rank $r$ and degree zero. Then there is
a holomorphic vector bundle $E$ on $X$ of rank $2r$, and a holomorphic connection $D$
on $E$, such that
\begin{enumerate}
\item $S\, :=\, F\otimes K^{1/2}_X$ is a holomorphic subbundle of $E$, and

\item the triple $(E,\, S,\, D)$ is a stable $r$--oper. In particular, $E\,=\, J^1(S\otimes TX)$ by
Lemma \ref{lem1}.
\end{enumerate}
\end{proposition}

\begin{proof}
The stable vector bundle $F$ of degree zero admits a holomorphic connection \cite[p.~203, Proposition 19]{At},
\cite{We} (in fact $F$ admits a unique holomorphic connection whose
monodromy representation is unitary \cite{NS}). Fix a holomorphic connection
$D^F$ on $F$. Now consider the holomorphic vector bundle 
$$
E\,:=\, F\otimes J^1(K^{-1/2}_X)\, .
$$
Since $J^1(K^{-1/2}_X)$ is indecomposable of degree zero, it admits a holomorphic
connection \cite[p.~203, Proposition 19]{At}, \cite{We}. In fact,
a projective structure on $X$ produces a holomorphic connection on $J^1(K^{-1/2}_X)$ \cite{Gu};
the definition of projective structure is recalled in Section \ref{se9.1}. Let $D^J$ be
a holomorphic connection on $J^1(K^{-1/2}_X)$. Consider the natural inclusion map
\begin{equation}\label{ei}
K^{1/2}_X\,=\, K^{-1/2}_X\otimes K_X\, \xhookrightarrow{\iota}\, J^1(K^{-1/2}_X)
\end{equation}
(see \eqref{ej}). Let
\begin{equation}\label{edj}
D^J(K^{1/2}_X)\, :\, K^{1/2}_X \, \longrightarrow\, K^{-1/2}_X\otimes K_X\,=\, K^{1/2}_X
\end{equation}
be the second fundamental form of this subbundle $K^{1/2}_X\, \hookrightarrow\,
J^1(K^{-1/2}_X)$ for the above connection $D^J$ on $J^1(K^{-1/2}_X)$. This homomorphism
$D^J(K^{1/2}_X)$ in \eqref{edj} is a nonzero constant scalar multiple of the identity map of $K^{1/2}_X$;
the scalar is nonzero because $K^{1/2}_X$ does not admit any holomorphic connection as its degree is nonzero.

The holomorphic connections
$D^F$ and $D^J$ together produce the holomorphic connection
\begin{equation}\label{edj2}
D\, :=\, D^F\otimes{\rm Id}_{J^1(K^{-1/2}_X)}+ {\rm Id}_F\otimes D^J
\end{equation}
on $F\otimes J^1(K^{-1/2}_X)\,=\, E$. The inclusion map $\iota$ in \eqref{ei}
produces an inclusion map
$$
S\, :=\, F\otimes K^{1/2}_X \, 
\xhookrightarrow{{\rm Id}_F\otimes\iota}\, F\otimes J^1(K^{-1/2}_X)\,=\, E\, .
$$
The second fundamental form $D_S$ (see \eqref{e1}) of $S$ for the connection $D$ in \eqref{edj2} coincides with
$$
\text{Id}_F\otimes D^J(K^{1/2}_X)\, :\, F\otimes K^{1/2}_X\, \longrightarrow\,
F\otimes (K^{-1/2}_X\otimes K_X)\,=\, F\otimes K^{1/2}_X\, ,
$$
where $D^J(K^{1/2}_X)$ is the homomorphism in \eqref{edj}.
Since $D^J(K^{1/2}_X)$ is a nonzero scalar multiple of the identity
map of $K^{1/2}_X$, we now conclude that $D_S$ is a nonzero scalar multiple of the identity map of
$F\otimes K^{1/2}_X$, in particular, $D_S$ is an isomorphism. Therefore, the triple $(E,\, S,\, D)$ is
a stable $r$--oper on $X$.
\end{proof}

\begin{proposition}\label{prop-1}
Let $F$ be a stable vector bundle on $X$ of degree zero, and let $V$ be any holomorphic
vector bundle on $X$. Then there is a holomorphic isomorphism
$$
h\, :\,F\otimes J^1(V)\, \longrightarrow\, J^1(F\otimes V)
$$
such that the exact sequence in \eqref{ej} for $W\,=\, F\otimes V$ coincides with
the exact sequence in \eqref{ej} for $W\,=\,V$ tensored with $F$.
\end{proposition}

\begin{proof}
As in the proof of Proposition \ref{prop1}, fix a holomorphic connection
$D^F$ on $F$. For any point $x\, \in\, X$, and any $v\, \in\, F_x$, let $\widehat{v}$
be the unique flat section of $F$, defined on a simply connected open neighborhood of
$x\, \in\, X$, such that $\widehat{v}(x)\,=\, x$ (see the proof of Lemma \ref{lem1}).
Take any element $w\, \in\, J^1(V)_x$, so $w$ is a section of $V$ defined on the first
order infinitesimal neighborhood of $x$. Therefore, $\widehat{v}'\otimes w$ 
is a section of $F\otimes V$ defined on the first order infinitesimal neighborhood of $x$,
where $\widehat{v}'$ is the restriction of $\widehat{v}$ to the first
order infinitesimal neighborhood of $x$. Now we have a map
$$
h(x)\, :\, F_x\otimes J^1(V)_x\, \longrightarrow\, J^1(F\otimes V)_x
$$
that sends any $v\otimes w\, \in\, F_x\otimes J^1(V)_x$ to the element $\widehat{v}'\otimes w$ 
of $J^1(F\otimes V)_x$ constructed above. It is straightforward to check that $h(x)$ is an
isomorphism. Moreover, we get a holomorphic isomorphism 
$$
h\, :\,F\otimes J^1(V) \, \longrightarrow\, J^1(F\otimes V)
$$
which coincides with $h(x)$ for any $x\,\in\, X$.

{}From the above construction of $h$ we conclude that
the following diagram of homomorphisms is commutative
\begin{equation}\label{cdn}
\begin{matrix}
0 &\longrightarrow & F\otimes (V\otimes K_X) &\longrightarrow & F\otimes J^1(V) & \longrightarrow &
F\otimes V &\longrightarrow & 0\\
&& \Big\Vert && \,\,\Big\downarrow h && \Big\Vert\\
0 &\longrightarrow & (F\otimes V)\otimes K_X &\longrightarrow & J^1(F\otimes V) &\longrightarrow &
F\otimes V &\longrightarrow & 0
\end{matrix}
\end{equation}
where the top one is the exact sequence in \eqref{ej} for $W\,=\,V$ tensored with $F$ and the
bottom one is the exact sequence in \eqref{ej} for $W\,=\, F\otimes V$. This completes the proof.
\end{proof}

For $F$ as in Proposition \ref{prop-1}, consider the short exact sequence
$$
0 \, \longrightarrow\, F\otimes K^{-1/2}_X\otimes K_X\,=\,F\otimes K^{1/2}_X
\,\longrightarrow\, J^1(F\otimes K^{-1/2}_X) \, \longrightarrow\, F\otimes K^{-1/2}_X \, \longrightarrow\,0
$$
in \eqref{ej} for $W\,=\, F\otimes K^{-1/2}_X$. It corresponds to an extension class
\begin{equation}\label{c1}
c(F)\,\in\, H^1(X,\, \text{Hom}(F\otimes K^{-1/2}_X,\, F\otimes K^{1/2}_X))\,=\, 
H^1(X,\, \text{End}(F)\otimes K_X)\,.
\end{equation}
Since $H^1(X,\, K_X)\,=\, H^0(X,\, {\mathcal O}_X)^*\,=\, {\mathbb C}^*\,=\, \mathbb C$, we have
\begin{equation}\label{c2}
\text{Id}_F\otimes 1\, \in\, H^0(X,\, \text{End}(F))\otimes H^1(X,\, K_X)\, \hookrightarrow\,
H^1(X,\, \text{End}(F)\otimes K_X)\, ;
\end{equation}
the above map $H^0(X,\, \text{End}(F))\otimes H^1(X,\, K_X)\, \longrightarrow\,
H^1(X,\, \text{End}(F)\otimes K_X)$ is injective because $H^0(X,\, \text{End}(F))\,=\, {\mathbb C}$.

\begin{corollary}\label{cor-1}
Let $F$ be a stable vector bundle on $X$ of degree zero. Then the cohomology class $c(F)$ in \eqref{c1}
coincides with ${\rm Id}_F\otimes 1$ in \eqref{c2}.
\end{corollary}

\begin{proof}
Consider the short exact sequence
$$
0 \, \longrightarrow\, K^{-1/2}_X\otimes K_X\,=\, K^{1/2}_X
\,\longrightarrow\, J^1(K^{-1/2}_X) \, \longrightarrow\, K^{-1/2}_X \, \longrightarrow\,0
$$
in \eqref{ej} for $W\,=\, K^{-1/2}_X$. The corresponding extension class coincides with
$$
1\, \in\, H^1(X,\, \text{Hom}(K^{-1/2}_X,\, K^{1/2}_X))\,=\, H^1(X,\, K_X)\,=\, {\mathbb C}
$$
\cite{Gu}. Hence the extension class of the tensor product of the above exact sequence with $F$
$$
0 \, \longrightarrow\, F\otimes K^{1/2}_X
\,\longrightarrow\, F\otimes J^1(K^{-1/2}_X) \, \longrightarrow\, F\otimes K^{-1/2}_X \, \longrightarrow\,0
$$
is $\text{Id}_F\otimes 1 \, \in\, H^1(X,\, \text{End}(F)\otimes K_X)\,=\, H^1(X,\,\text{Hom}(F\otimes K^{-1/2}_X,\,
F\otimes K^{1/2}_X))$ in \eqref{c2}. Therefore, the proof is completed using the isomorphism $h$ in
\eqref{cdn} together with the commutativity of the diagram in \eqref{cdn}.
\end{proof}

\subsection{An equivalence relation}\label{sec2.2}

Take a holomorphic vector bundle $F$ on $X$ of rank $r$ and degree zero; it need not be stable. Assume
that the jet bundle $J^1(F\otimes K^{-1/2}_X)$ has a holomorphic connection $D_1$. Consider the subbundle
$$
F\otimes K^{1/2}_X\,=\,F\otimes K^{-1/2}_X\otimes K_X\, \subset\, J^1(F\otimes K^{-1/2}_X)
$$
(see \eqref{ej}). The second fundamental form of this subbundle for the connection $D_1$ is a homomorphism
\begin{equation}\label{e7}
{\mathcal S}(F, D_1)\, :\, F\otimes K^{1/2}_X \, \longrightarrow\, (J^1(F\otimes K^{-1/2}_X)/(F\otimes K^{1/2}_X))
\otimes K_X\,=\, F\otimes K^{-1/2}_X\otimes K_X\,=\, F\otimes K^{1/2}_X
\end{equation}
(see \eqref{e1} and \eqref{ej}).

Now assume that $J^1(F\otimes K^{-1/2}_X)$ admits a holomorphic connection $D_1$ for which the
second fundamental form ${\mathcal S}(F, D_1)$ in \eqref{e7} is the identity map of $F\otimes K^{1/2}_X$.
Note that for such a holomorphic connection $D_1$, the triple $(J^1(F\otimes K^{-1/2}_X),\,
F\otimes K^{1/2}_X,\, D_1)$ is a $r$--oper (see Definition \ref{defro}, \eqref{ew} and \eqref{ew2}).

Let
\begin{equation}\label{c0}
{\rm Conn}(J^1(F\otimes K^{-1/2}_X))
\end{equation}
denote the space of all holomorphic connection on the vector bundle $J^1(F\otimes K^{-1/2}_X)$.

\begin{definition}\label{def2}
Let $$\widetilde{\mathcal C}(F)\, \subset\, {\rm Conn}(J^1(F\otimes K^{-1/2}_X))$$ be the
space of all holomorphic connections $D$ on
$J^1(F\otimes K^{-1/2}_X)$ with the property that the corresponding second fundamental form
${\mathcal S}(F, D)$ in \eqref{e7} is the identity map of $F\otimes K^{1/2}_X$.
\end{definition}

We note that $\widetilde{\mathcal C}(F)$ is nonempty by the assumption on $F$.

So $(J^1(F\otimes K^{-1/2}_X),\, F\otimes K^{1/2}_X,\, D)$ is a $r$--oper for every
$D\, \in\, \widetilde{\mathcal C}(F)$.

We shall see in Corollary \ref{cor1} that $\widetilde{\mathcal C}(F)\, =\,
{\rm Conn}(J^1(F\otimes K^{-1/2}_X))$ when the vector bundle $F$ is stable.

For notational convenience, define
\begin{equation}\label{e8}
S\,:=\, F\otimes K^{1/2}_X\ \ \text{ and }\ \ E\, :=\, J^1(F\otimes K^{-1/2}_X)
\end{equation}
(see \eqref{ew}, \eqref{ew2} and Lemma \ref{lem1}), so $S$ is a holomorphic subbundle of $E$ by \eqref{ej}
(see \eqref{e8b} below).
Note that we have $\text{End}(S)\,=\, \text{End}(F)\otimes\text{End}(K^{1/2}_X)\,=\, \text{End}(F)$.

Setting $W\,=\, F\otimes K^{-1/2}_X$ in \eqref{ej}, we get a short exact sequence
\begin{equation}\label{e8b}
0\, \longrightarrow\, F\otimes K^{-1/2}_X\otimes K_X\,=\, S\, \stackrel{\iota}{\longrightarrow}\, 
J^1(F\otimes K^{-1/2}_X)\,=\, E\, \stackrel{q_0}{\longrightarrow}\, F\otimes K^{-1/2}_X\,=\, S\otimes TX
\, \longrightarrow\, 0
\end{equation}
(see \eqref{e8}). Consequently, there is a natural inclusion map
\begin{equation}\label{e5}
\psi\, :\, H^0(X,\, \text{End}(S)\otimes K^{\otimes 2}_X)\,=\,
H^0(X,\, \text{Hom}(S\otimes TX,\, S\otimes K_X))
\, \hookrightarrow\, H^0(X,\, \text{End}(E)\otimes K_X)
\end{equation}
that sends any homomorphism $\alpha\, :\, S\otimes TX \, \longrightarrow\, S\otimes K_X$ to the following
composition of homomorphisms
$$
E \, \stackrel{q_0}{\longrightarrow}\, S\otimes TX \, \stackrel{\alpha}{\longrightarrow}\, S\otimes K_X
\, \xrightarrow{\iota\otimes{\rm Id}_{K_X}}\, E\otimes K_X\, ,
$$
where $q_0$ and $\iota$ are the homomorphisms in \eqref{e8b}.

The space of holomorphic connections on $E$ is an affine space modeled on
the vector space $H^0(X,\, \text{End}(E)\otimes K_X)$; recall that $E$ is assumed to admit a holomorphic
connection. In view of the homomorphism $\psi$ in \eqref{e5}, for any holomorphic connection $D$ on $E$ and
any $\alpha\, \in\, H^0(X,\, \text{End}(S)\otimes K^{\otimes 2}_X)$, we get a holomorphic connection
$D+\psi(\alpha)$ on $E$. If
$D\, \in\,\widetilde{\mathcal C}(F)$ (see Definition \ref{def2}), then it can be shown that
\begin{equation}\label{ejn}
D+\psi(\alpha)\, \in\, \widetilde{\mathcal C}(F)\,.
\end{equation}
Indeed, the restrictions of $D$ and $D+\psi(\alpha)$ to the subbundle $S$ in \eqref{e8b} coincide. Therefore,
the second fundamental form of $S$ for the connection $D$ coincides with the
second fundamental form of $S$ for the connection $D+\psi(\alpha)$. Hence \eqref{ejn}
holds. Consequently, the vector space
$H^0(X,\, \text{End}(S)\otimes K^{\otimes 2}_X)$ acts on $\widetilde{\mathcal C}(F)$. Let
\begin{equation}\label{ea}
\widetilde{\mathcal C}(F)\times H^0(X,\, \text{End}(S)\otimes K^{\otimes 2}_X)
\,\longrightarrow\, \widetilde{\mathcal C}(F)
\end{equation}
be this action. The action in \eqref{ea}
is evidently free; however, the action is not transitive (the dimension of
$H^0(X,\, \text{End}(S)\otimes K^{\otimes 2}_X)$ is smaller than that of $\widetilde{\mathcal C}(F)$).

Let $\text{ad}(S)\, \subset\, \text{End}(S)$ be the subbundle of rank $r^2-1$ given by the sheaf of
endomorphisms of $S$ of trace zero. Note that we have
\begin{equation}\label{e10}
\text{End}(S)\,=\, \text{ad}(S)\oplus {\mathcal O}_X\, ;
\end{equation}
the inclusion map ${\mathcal O}_X\, \hookrightarrow\, \text{End}(S)$ sends a locally defined
holomorphic function $f$ on $X$
to the locally defined endomorphism of $S$ that maps any locally defined section $s$ of $E$ to $f\cdot s$. It
is evident that $\text{ad}(S)\,=\, \text{ad}(F)$ (see \eqref{e8}). The decomposition of
$\text{End}(S)$ in \eqref{e10} produces a decomposition
\begin{equation}\label{e6}
H^0(X,\, \text{End}(S)\otimes K^{\otimes 2}_X)\, =\, H^0(X,\, \text{ad}(S)\otimes K^{\otimes 2}_X)
\oplus H^0(X,\,K^{\otimes 2}_X)\, .
\end{equation}
Consider the action of $H^0(X,\, \text{End}(S)\otimes K^{\otimes 2}_X)$ on $\widetilde{\mathcal C}(F)$ in \eqref{ea}.
In view of \eqref{e6}, from this action we obtain an action of $H^0(X,\, \text{ad}(S)\otimes K^{\otimes 2}_X)$ on
$\widetilde{\mathcal C}(F)$.

\begin{definition}\label{def3}
Define ${\mathcal C}(F)$ to be the quotient space $\widetilde{\mathcal C}(F)\big{/}
H^0(X,\, \text{ad}(S)\otimes K^{\otimes 2}_X)$ for the above action of $H^0(X,\,
\text{ad}(S)\otimes K^{\otimes 2}_X)$ on the space $\widetilde{\mathcal C}(F)$ in Definition \ref{def2}.
\end{definition}

{}From \eqref{e6} it follows immediately that $H^0(X,\,K^{\otimes 2}_X)$ acts freely on
${\mathcal C}(F)$. It is evident that
$$
{\mathcal C}(F)\big{/}H^0(X,\,K^{\otimes 2}_X)\,=\, \widetilde{\mathcal C}(F)\big{/}
H^0(X,\, \text{End}(S)\otimes K^{\otimes 2}_X)\, .
$$

We have a homomorphism
\begin{equation}\label{vp}
\varpi\, :\, \text{End}(F)\otimes K_X\,=\, \text{Hom}(F\otimes K^{-1/2}_X,\, F\otimes K^{1/2}_X) 
\,\longrightarrow\, \text{End}(J^1(F\otimes K^{-1/2}_X))
\end{equation}
that sends any locally defined homomorphism $\eta\, :\, F\otimes K^{-1/2}_X\, \longrightarrow\,
F\otimes K^{1/2}_X$ to the following composition of (locally defined) homomorphisms:
$$
J^1(F\otimes K^{-1/2}_X)\, \stackrel{q_0}{\longrightarrow}\, F\otimes K^{-1/2}_X
\, \stackrel{\eta}{\longrightarrow}\, F\otimes K^{1/2}_X \,\stackrel{\iota}{\longrightarrow}\,
J^1(F\otimes K^{-1/2}_X)\, ,
$$
where $\iota$ and $q_0$ are the homomorphisms in \eqref{e8b}.

\begin{lemma}\label{lem4}
Let $F$ be a stable vector bundle on $X$ of rank $r$ and degree zero. The homomorphism
$$
\widehat{\varpi}\, :\, H^0(X,\, {\rm End}(F)\otimes K_X)\oplus {\mathbb C}\,\longrightarrow\,
H^0\big(X,\, {\rm End}(J^1(F\otimes K^{-1/2}_X))\big)\, ,
$$
that sends any $(v,\, c)\, \in\, H^0(X,\, {\rm End}(F)\otimes K_X)\oplus {\mathbb C}$ to
$\varpi(v)+c\cdot{\rm Id}_{J^1(F\otimes K^{-1/2}_X)}$, where $\varpi$ is
the homomorphism in \eqref{vp}, is an isomorphism.
\end{lemma}

\begin{proof}
The holomorphic vector bundle $F\otimes K^{-1/2}_X$ does not admit any holomorphic connection because
$\text{degree}(F\otimes K^{-1/2}_X)\,=\, r(1-g)\, \not=\, 0$ \cite[p.~202, Proposition 18(i)]{At}
(recall that $g\, \geq\, 2$). The
statement that $F\otimes K^{-1/2}_X$ does not admit any holomorphic connection is equivalent to the
statement that the short exact sequence in \eqref{e8b} does not split holomorphically \cite{At}.
Indeed, a holomorphic splitting homomorphism $$D_1\, :\,J^1(F\otimes K^{-1/2}_X)\, \longrightarrow\, 
(F\otimes K^{-1/2}_X)\otimes K_X$$ for \eqref{e8b} such that $D_1\circ \iota\,=\, \text{Id}$,
where $\iota$ is the homomorphism in \eqref{e8b}, defines a holomorphic
differential operator of order one $\widetilde{D}_1\, :\, F\otimes K^{-1/2}_X\, \longrightarrow\,
(F\otimes K^{-1/2}_X)\otimes K_X$ which satisfies the Leibniz identity, in other words,
$\widetilde{D}_1$ is a holomorphic connection on $F\otimes K^{-1/2}_X$; conversely,
the homomorphism $D_2\, :\,J^1(F\otimes K^{-1/2}_X)\, \longrightarrow\,
(F\otimes K^{-1/2}_X)\otimes K_X$ corresponding to any holomorphic connection on $F\otimes K^{-1/2}_X$
produces a holomorphic splitting of \eqref{e8b}.

Take any endomorphism $\rho\, :\, J^1(F\otimes K^{-1/2}_X)\, \longrightarrow\, J^1(F\otimes K^{-1/2}_X)$
over $X$.
Consider the composition of homomorphisms
\begin{equation}\label{ch1}
F\otimes K^{1/2}_X\,\stackrel{\iota}{\longrightarrow}\, J^1(F\otimes K^{-1/2}_X) 
\,\stackrel{\rho}{\longrightarrow}\, J^1(F\otimes K^{-1/2}_X)
\, \stackrel{q_0}{\longrightarrow}\, F\otimes K^{-1/2}_X\, ,
\end{equation}
where $\iota$ and $q_0$ are the homomorphisms in \eqref{e8b}. Both the vector bundles
$F\otimes K^{1/2}_X$ and $F\otimes K^{-1/2}_X$ are stable because $F$ is so. We have
$$
\frac{\text{degree}(F\otimes K^{1/2}_X)}{\text{rank}(F\otimes K^{1/2}_X)}\, >\,
\frac{\text{degree}(F\otimes K^{-1/2}_X)}{\text{rank}(F\otimes K^{-1/2}_X)}$$ since $g \, \geq\, 2$.
Therefore, there is no nonzero homomorphism from $F\otimes K^{1/2}_X$ to $F\otimes K^{-1/2}_X$.
In particular, the composition of homomorphisms in \eqref{ch1} vanishes identically.

Now let $\rho'\, :=\, \rho\big\vert_{F\otimes K^{1/2}_X}\, :\, F\otimes K^{1/2}_X\, \longrightarrow\, 
F\otimes K^{1/2}_X$ be the restriction of $\rho$ to the subbundle $F\otimes K^{1/2}_X$
in \eqref{e8b}. Since
$F\otimes K^{1/2}_X$ is stable, there is a $c\, \in\, \mathbb C$ such that
$\rho'\,=\, c\cdot {\rm Id}_{F\otimes K^{1/2}_X}$. Define the endomorphism over $X$
\begin{equation}\label{r1}
\rho_1\, :=\, \rho-c\cdot {\rm Id}_{J^1(F\otimes K^{-1/2}_X)}\,\, :\,\,
J^1(F\otimes K^{-1/2}_X)\, \longrightarrow\, J^1(F\otimes K^{-1/2}_X)\, .
\end{equation}
So we have $\rho_1(F\otimes K^{1/2}_X)\,=\, 0$. This implies that there is a homomorphism
$$
\rho'_2\, \,:\,\, J^1(F\otimes K^{-1/2}_X)\big{/}(F\otimes K^{1/2}_X)\,=\, F\otimes K^{-1/2}_X\,
\longrightarrow\, J^1(F\otimes K^{-1/2}_X)
$$
such that
\begin{equation}\label{ji}
\rho_1\,=\, \rho'_2\circ q_0\, ,
\end{equation}
where $q_0$ and $\rho_1$ are the homomorphisms in \eqref{e8b} and \eqref{r1} respectively.
Now $q_0\circ \rho'_2\, \in\, H^0(X,\, \text{End}(F\otimes K^{-1/2}_X))$ coincides with
$c'\cdot\text{Id}$ for some $c'\, \in\, \mathbb C$, because $F\otimes K^{-1/2}_X$ is stable.
Next we note that if $c'\, \not=\, 0$, then $\frac{1}{c'}\rho'_2\, :\, F\otimes K^{-1/2}_X\,
\longrightarrow\, J^1(F\otimes K^{-1/2}_X)$ is
a holomorphic splitting of the short exact sequence in \eqref{e8b}.
Since the short exact sequence in \eqref{e8b} does not admit a holomorphic splitting (this
was shown earlier), we conclude that $c'\,=\, 0$.

Given that $c'\,=\, 0$, it is deduced that there is a homomorphism
$$\rho_2\,\, :\,\, F\otimes K^{-1/2}_X\, \longrightarrow\, F\otimes K^{1/2}_X$$
such that $\rho'_2\,=\,\iota\circ\rho_2$, where $\iota$ and $\rho'_2$ are the homomorphisms in \eqref{e8b}
and \eqref{ji} respectively. Therefore, from \eqref{ji} and the definition of $\varpi$ in \eqref{vp} it follows that
$$
\rho_1\,=\, \iota\circ\rho_2\circ q_0\,=\, \varpi(\rho_2)\, .
$$
So from \eqref{r1} we know that $\rho\,=\, \varpi(\rho_2)+c\cdot{\rm Id}_{J^1(F\otimes K^{-1/2}_X)}$.
\end{proof}

As in Lemma \ref{lem4}, $F$ is a stable vector bundle on $X$ of rank $r$ and degree zero. Let
$$
\text{Aut}(J^1(F\otimes K^{-1/2}_X))
$$
be the group of all holomorphic automorphisms of $J^1(F\otimes K^{-1/2}_X)$; it is a complex
affine algebraic group, in fact, it is a Zariski open subset of the
affine space $H^0\big(X,\, {\rm End}(J^1(F\otimes K^{-1/2}_X))\big)$.

{}From Lemma \ref{lem4} it follows immediately that
\begin{equation}\label{aj}
\text{Aut}(J^1(F\otimes K^{-1/2}_X))\,=\, H^0(X,\, {\rm End}(F)\otimes K_X)\times {\mathbb C}^\star\, ,
\end{equation}
where ${\mathbb C}^\star\,=\, {\mathbb C}\setminus\{0\}$ is the multiplicative group.

The group $\text{Aut}(J^1(F\otimes K^{-1/2}_X))$ has a natural action on ${\rm Conn}(J^1(F\otimes
K^{-1/2}_X))$ defined in \eqref{c0}. The action of any $T\, \in\, \text{Aut}(J^1(F\otimes K^{-1/2}_X))$
sends any holomorphic connection $D$ to the holomorphic connection given by the composition
$(T\otimes {\rm Id}_{K_X})\circ D\circ T^{-1}$ of operators.

The isomorphism in \eqref{aj} has the following consequence:

\begin{corollary}\label{cor3}
Let $F$ be a stable vector bundle on $X$ of rank $r$ and degree zero.
The action of ${\rm Aut}(J^1(F\otimes K^{-1/2}_X))$ on ${\rm Conn}(J^1(F\otimes K^{-1/2}_X))$
preserves the subset $\widetilde{\mathcal C}(F)$ in Definition \ref{def2}.

The action of ${\rm Aut}(J^1(F\otimes K^{-1/2}_X))$ on $\widetilde{\mathcal C}(F)$ descends to
an action of ${\rm Aut}(J^1(F\otimes K^{-1/2}_X))$ on the quotient space ${\mathcal C}(F)$
in Definition \ref{def3}.
\end{corollary}

\begin{proof}
The first statement is straightforward. Take any $D\,\in\, {\rm Conn}(J^1(F\otimes K^{-1/2}_X))$
and $T\,\in\,{\rm Aut}(J^1(F\otimes K^{-1/2}_X))$. From \eqref{aj} it follows immediately that the
second fundamental forms of the subbundle $F\otimes K^{1/2}_X\, \subset\, J^1(F\otimes K^{-1/2}_X)$
(see \eqref{e8b}) for the two connections $D$ and $(T\otimes {\rm Id}_{K_X})\circ D\circ T^{-1}$
coincide; see also Corollary \ref{cor1}. This implies that the action of ${\rm Aut}(J^1(F\otimes
K^{-1/2}_X))$ on ${\rm Conn}(J^1(F\otimes K^{-1/2}_X))$ preserves $\widetilde{\mathcal C}(F)$.

Take any $D\,\in\, {\rm Conn}(J^1(F\otimes K^{-1/2}_X))$,\, $T\,\in\,
{\rm Aut}(J^1(F\otimes K^{-1/2}_X))$ and $$B\, \in\, H^0(X,\, \text{ad}(S)\otimes K^{\otimes 2}_X).$$ From
\eqref{aj} it follows that
$$
(T\otimes {\rm Id}_{K_X})\circ (D+B)\circ T^{-1}\,=\, (T\otimes {\rm Id}_{K_X})\circ D\circ T^{-1} +B\, .
$$
Therefore, the translation action of
$H^0(X,\, \text{ad}(S)\otimes K^{\otimes 2}_X)$ on $\widetilde{\mathcal C}(F)$ and the action of
${\rm Aut}(J^1(F\otimes K^{-1/2}_X))$ on $\widetilde{\mathcal C}(F)$ commute.
The second statement of the corollary follows from this.
\end{proof}

In the next section we will put structures on the quotient space obtained from Corollary \ref{cor3}
\begin{equation}\label{qe}
{\mathcal D}(F)\, :=\, {\mathcal C}(F)\big{/}{\rm Aut}(J^1(F\otimes K^{-1/2}_X))\, .
\end{equation}

\section{The space of holomorphic connections}\label{sec3}

Take a stable holomorphic vector bundle $F$ on $X$ of rank $r$ and degree zero.

As before, $K^{1/2}_X$ is a theta characteristic on $X$. From Proposition
\ref{prop1} and Lemma \ref{lem1b} we know that 
$J^1(F\otimes K^{-1/2}_X)$ admits a holomorphic connection $D_1$ for which the second fundamental form ${\mathcal 
S}(F, D_1)$ in \eqref{e7} is the identity map of $F\otimes K^{1/2}_X$.

Recall that the space ${\rm Conn}(J^1(F\otimes K^{-1/2}_X))$ in \eqref{c0} is an affine space modeled
on the vector space $H^0(X, \,\text{End}(J^1(F\otimes K^{-1/2}_X))\otimes K_X)$. Let
\begin{equation}\label{e21}
0\, \longrightarrow\, 
F\otimes K^{3/2}_X\,=\, F\otimes K^{1/2}_X\otimes K_X \,\xrightarrow{\iota\otimes{\rm Id}_{K_X}}\,
J^1(F\otimes K^{-1/2}_X)\otimes K_X
\end{equation}
$$
\xrightarrow{q_0\otimes{\rm Id}_{K_X}}\,
F\otimes K^{-1/2}_X\otimes K_X \,=\, F\otimes K^{1/2}_X\, \longrightarrow\, 0
$$
be the short exact sequence of vector bundles obtained by tensoring the short exact sequence in \eqref{e8b} by $K_X$.

\begin{lemma}\label{lem3}
Take any
$$
\Psi\, \in\, H^0(X, \,{\rm Hom}(J^1(F\otimes K^{-1/2}_X),\, J^1(F\otimes K^{-1/2}_X)\otimes K_X))\, .
$$
Then
$$
\Psi(F\otimes K^{1/2}_X)\, \subset\, F\otimes K^{3/2}_X\, ,
$$
where $F\otimes K^{1/2}_X$ (respectively, $F\otimes K^{3/2}_X$) is the subbundle of
$J^1(F\otimes K^{-1/2}_X)$ (respectively, $J^1(F\otimes K^{-1/2}_X)\otimes K_X$) in
\eqref{e8b} (respectively, \eqref{e21}).
\end{lemma}

\begin{proof}
Take any holomorphic homomorphism $$\Psi\, :\, J^1(F\otimes K^{-1/2}_X)\,\longrightarrow\,
J^1(F\otimes K^{-1/2}_X)\otimes K_X\, .$$
Consider the composition of homomorphisms
\begin{equation}\label{e22}
F\otimes K^{1/2}_X\, \hookrightarrow\, J^1(F\otimes K^{-1/2}_X) \,\stackrel{\Psi}{\longrightarrow}\,
J^1(F\otimes K^{-1/2}_X)\otimes K_X
\end{equation}
$$
\xrightarrow{q_0\otimes{\rm Id}_{K_X}}\,
F\otimes K^{-1/2}_X\otimes K_X \,=\, F\otimes K^{1/2}_X\, ,
$$
where $q_0\otimes{\rm Id}_{K_X}$ is the homomorphism in \eqref{e21}. Since the vector bundle $F$, and hence
$F\otimes K^{1/2}_X$, is stable, any nonzero holomorphic endomorphism of $F\otimes K^{1/2}_X$ is
an isomorphism.

We will now show that the composition of homomorphisms in \eqref{e22} is the zero
homomorphism.

To prove this by contradiction,
assume that the composition of homomorphisms in \eqref{e22} is nonzero. As observed above, this implies
that the composition of homomorphisms in \eqref{e22} is an isomorphism. Consequently,
$\Psi(F\otimes K^{1/2}_X)$ is a subbundle of $J^1(F\otimes K^{-1/2}_X)\otimes K_X$.
Now consider the subbundle
\begin{equation}\label{el1}
\Psi(F\otimes K^{1/2}_X)\otimes TX\, \subset\, J^1(F\otimes K^{-1/2}_X)\otimes K_X\otimes TX
\,=\, J^1(F\otimes K^{-1/2}_X)\, .
\end{equation}
The condition that the composition of homomorphisms in \eqref{e22} is an isomorphism implies that
the homomorphism
$$
\Psi(F\otimes K^{1/2}_X)\otimes TX \, \stackrel{q'_0}{\longrightarrow}\, F\otimes K^{-1/2}_X\, ,
$$
where $q'_0$ is the restriction of the projection $q_0$
in \eqref{e8b} (see \eqref{el1}), is an isomorphism. Consequently, the subbundle
$\Psi(F\otimes K^{1/2}_X)\otimes TX\, \subset\, J^1(F\otimes K^{-1/2}_X)$ in \eqref{el1} produces a holomorphic
splitting of the short exact sequence in \eqref{e8b}.
But it was observed in the proof of Lemma \ref{lem4} that the short exact sequence in \eqref{e8b}
does not split holomorphically.

In view of the above contradiction we conclude that
the composition of homomorphisms in \eqref{e22} is the zero homomorphism.
This proves the lemma.
\end{proof}

The following two results are deduced using Lemma \ref{lem3}.

\begin{corollary}\label{cor1}
As in Lemma \ref{lem3}, $F$ is a stable vector bundle on $X$ of degree zero.
The space $\widetilde{\mathcal C}(F)$ in Definition \ref{def2} actually coincides with
${\rm Conn}(J^1(F\otimes K^{-1/2}_X))$ in \eqref{c0}.
\end{corollary}

\begin{proof}
Take any holomorphic connection $D$ on $J^1(F\otimes K^{-1/2}_X)$ and any
$$
\Psi\, \in\, H^0(X, \,{\rm Hom}(J^1(F\otimes K^{-1/2}_X),\, J^1(F\otimes K^{-1/2}_X)\otimes K_X))\, .
$$
So $D+\Psi$ is a holomorphic connection on $J^1(F\otimes K^{-1/2}_X)$. Consider the second fundamental
forms of the subbundle $F\otimes K^{1/2}_X\, \subset\, J^1(F\otimes K^{-1/2}_X)$ for the two connections
$D$ and $D+\Psi$. From Lemma \ref{lem3} it follows immediately that these two
second fundamental forms actually coincide. Since
$J^1(F\otimes K^{-1/2}_X)$ admits a holomorphic connection $D_1$ for which the second fundamental form ${\mathcal
S}(F, D_1)$ in \eqref{e7} is the identity map of $F\otimes K^{1/2}_X$ (see Proposition
\ref{prop1}), the proof is complete.
\end{proof}

\begin{lemma}\label{lem5}
The vector space $H^0\big(X, \,{\rm Hom}(J^1(F\otimes K^{-1/2}_X),\, J^1(F\otimes K^{-1/2}_X)\otimes K_X)\big)$
fits in the following short exact sequence:
$$
0\, \longrightarrow\, H^0(X,\, {\rm End}(F)\otimes K^{\otimes 2}_X)
\, \longrightarrow\, H^0(X, \,{\rm Hom}(J^1(F\otimes K^{-1/2}_X),\, J^1(F\otimes K^{-1/2}_X)\otimes K_X))
$$
$$
\, \longrightarrow\, H^0(X,\, {\rm End}(F)\otimes K_X)^{\oplus 2}\, \longrightarrow\, 0\, .
$$
\end{lemma}

\begin{proof}
Let
$$
{\rm Hom}_P(J^1(F\otimes K^{-1/2}_X),\, J^1(F\otimes K^{-1/2}_X)\otimes K_X)\, \subset\,
{\rm Hom}(J^1(F\otimes K^{-1/2}_X),\, J^1(F\otimes K^{-1/2}_X)\otimes K_X)
$$
be the subbundle defined by the sheaf of homomorphisms from $J^1(F\otimes K^{-1/2}_X)$ to
$J^1(F\otimes K^{-1/2}_X)\otimes K_X$ that take the subbundle 
$F\otimes K^{1/2}_X$ in \eqref{e8b} to the subbundle $F\otimes K^{3/2}_X$ in \eqref{e21}. Let
\begin{gather}
\text{Hom}(F\otimes K^{-1/2}_X,\, F\otimes K^{3/2}_X)\,=\,
{\rm End}(F)\otimes K^{\otimes 2}_X\nonumber\\
\stackrel{\bf j}{\hookrightarrow} \, 
{\rm Hom}_P(J^1(F\otimes K^{-1/2}_X),\, J^1(F\otimes K^{-1/2}_X)\otimes K_X)\label{jl}
\end{gather}
be the homomorphism that sends any locally defined homomorphism
$$
\eta\, :\, F\otimes K^{-1/2}_X\, \longrightarrow\, F\otimes K^{3/2}_X
$$
to the following composition of (locally defined) homomorphisms:
$$
J^1(F\otimes K^{-1/2}_X)\, \stackrel{q_0}{\longrightarrow}\, F\otimes K^{-1/2}_X\,
\stackrel{\eta}{\longrightarrow}\, F\otimes K^{3/2}_X \,\xrightarrow{\iota\otimes{\rm Id}_{K_X}}\,
J^1(F\otimes K^{-1/2}_X)\otimes K_X\, ,
$$
where $q_0$ and $\iota\otimes{\rm Id}_{K_X}$ are the homomorphisms in \eqref{e8b} and \eqref{e21} respectively.

We have a natural surjective homomorphism
$$
\varpi_1\, :\, {\rm Hom}_P(J^1(F\otimes K^{-1/2}_X),\, J^1(F\otimes K^{-1/2}_X)\otimes K_X)
$$
$$
\longrightarrow\, \text{Hom}(F\otimes K^{1/2}_X,\, F\otimes K^{3/2}_X)\oplus
\text{Hom}(F\otimes K^{-1/2}_X,\, F\otimes K^{1/2}_X)\,=\, (\text{End}(F)\otimes K_X)^{\oplus 2}
$$
that sends any homomorphism to the induced homomorphism from the subbundle (respectively, quotient bundle)
in \eqref{e8b} to the subbundle (respectively, quotient bundle) in \eqref{e21}.
Consequently, we have a short exact sequence of holomorphic vector bundles on $X$
\begin{gather}
0 \, \longrightarrow\,{\rm End}(F)\otimes K^{\otimes 2}_X\, \stackrel{\bf j}{\longrightarrow} \,
{\rm Hom}_P(J^1(F\otimes K^{-1/2}_X),\, J^1(F\otimes K^{-1/2}_X)\otimes K_X)\nonumber\\
\,\stackrel{\varpi_1}{\longrightarrow}\, (\text{End}(F)\otimes K_X)^{\oplus 2} \, \longrightarrow\, 0\, ,\label{ec1}
\end{gather}
where $\bf j$ is the homomorphism in \eqref{jl}.

Let
\begin{gather}
0 \, \longrightarrow\,H^0(X,\, {\rm End}(F)\otimes K^{\otimes 2}_X)\, \longrightarrow \,
H^0(X,\, {\rm Hom}_P(J^1(F\otimes K^{-1/2}_X),\, J^1(F\otimes K^{-1/2}_X)\otimes K_X))\nonumber\\
\longrightarrow \, H^0(X,\, (\text{End}(K)\otimes K_X)^{\oplus 2})
\, \longrightarrow\,H^1(X,\, {\rm End}(F)\otimes K^{\otimes 2}_X)\label{lj2}
\end{gather}
be the long exact sequence of cohomologies corresponding to the exact sequence in \eqref{ec1}. By Serre duality,
\begin{equation}\label{lj3}
H^1(X,\, {\rm End}(F)\otimes K^{\otimes 2}_X)\,=\, H^0(X,\, {\rm End}(F)\otimes TX)^*\,=\, 0\, ,
\end{equation}
because ${\rm End}(F)$ is semistable of degree zero (recall that $F$ is stable) and $\text{degree}(TX)\, <\, 0$
(recall that $g\, \geq\, 2$); the unitary flat connection on $F$, \cite{NS}, induces a
unitary flat connection on ${\rm End}(F)$ and hence ${\rm End}(F)$ is polystable, in particular,
${\rm End}(F)$ is semistable. From Lemma \ref{lem3} we have
$$
H^0\left(X,\, {\rm Hom}_P(J^1(F\otimes K^{-1/2}_X),\, J^1(F\otimes K^{-1/2}_X)\otimes K_X)\right)
$$
$$
=\,
H^0\left(X,\, {\rm Hom}(J^1(F\otimes K^{-1/2}_X),\, J^1(F\otimes K^{-1/2}_X)\otimes K_X)\right)\, .
$$
Consequently, the lemma follows from \eqref{lj2} and \eqref{lj3}. 
\end{proof}

\section{Infinitesimal deformations}\label{sec4}

The space of all infinitesimal deformations of a compact Riemann surface $X$ is identified with
$H^1(X,\,TX)$. By Serre duality,
$$
H^1(X,\,TX)^*\,=\, H^0(X,\, K^{\otimes 2}_X)\, .
$$
For a holomorphic vector bundle $V$ on $X$, the infinitesimal deformations of $V$,
keeping $X$ fixed, are parametrized by $H^1(X,\,\text{End}(V))$. By Serre duality,
$$
H^1(X,\,\text{End}(V))^*\,=\, H^0(X,\,\text{End}(V)\otimes K_X)\, .
$$

Consider the holomorphic vector bundle
\begin{equation}\label{e13}
\text{Diff}^i_X(V,\, V)\, :=\, \text{Hom}(J^i(V),\, V)
\end{equation}
on $X$
given by the sheaf of all holomorphic differential operators of order $i$ from $V$ to itself. Take the
dual of the exact sequence in \eqref{ej} for $W\,=\, V$
$$
0 \, \longrightarrow\, V^* \, \stackrel{\iota}{\longrightarrow}\, J^1(V)^* \,
\longrightarrow\, V^*\otimes TX \, \longrightarrow\,0\, .
$$
Tensoring it with $V$ the following 
short exact sequence of holomorphic vector bundles on $X$ is obtained
\begin{equation}\label{e9}
0\, \longrightarrow\, \text{Diff}^0_X(V,\, V)\,=\, \text{End}(V)\, \longrightarrow\, \text{Diff}^1_X(V,\, V)
\, \stackrel{\sigma_0}{\longrightarrow}\, \text{End}(V)\otimes TX \, \longrightarrow\, 0\, ;
\end{equation}
the above projection $\sigma_0$ coincides with the symbol map. Using the natural inclusion ${\mathcal O}_X\, \subset\,
\text{End}(V)$ (see \eqref{e10}), we have $TX\, \subset\,\text{End}(V)\otimes TX$. Now define the Atiyah bundle for $V$
\begin{equation}\label{e12}
\text{At}(V)\, :=\, \sigma^{-1}_0(TX) \, \subset\, \text{Diff}^1_X(V,\, V)\, ,
\end{equation}
where $\sigma_0$ is the projection in \eqref{e9} \cite{At}. The exact sequence in \eqref{e9} produces the
short exact sequence of holomorphic vector bundles on $X$
\begin{equation}\label{e11}
0\, \longrightarrow\, \text{Diff}^0_X(V,\, V)\,=\, \text{End}(V)\, \longrightarrow\, \text{At}(V)
\, \stackrel{\sigma}{\longrightarrow}\, TX \, \longrightarrow\, 0\, ,
\end{equation}
where $\sigma$ is the restriction of $\sigma_0$ (in \eqref{e9}) to the subbundle $\text{At}(V)$; this exact
sequence is known as the Atiyah exact sequence for $V$ (see \cite{At}).

The space of all infinitesimal deformations of the pair $(X,\, V)$ is known to be identified with
$H^1(X,\, \text{At}(V))$ (see \cite[p.~1413, Proposition 4.3]{Ch}, \cite[p.~127, (2.12)]{BHH}).
The homomorphism
\begin{equation}\label{e11b}
H^1(X,\, \text{At}(V))\, \longrightarrow\, H^1(X,\, TX)
\end{equation}
produced by the projection
$\sigma$ in \eqref{e11} coincides with the forgetful map that sends an infinitesimal deformation of
$(X,\, V)$ to the infinitesimal deformation of $X$ obtained from it by simply forgetting the vector bundle.
The homomorphism
$$H^1(X,\, \text{End}(V))\, \longrightarrow\, H^1(X,\, \text{At}(V))$$
induced by the homomorphism
$\text{End}(V)\, \longrightarrow\, \text{At}(V)$ in \eqref{e11} coincides with the map
that sends an infinitesimal deformation of $V$ to the infinitesimal deformation of
$(X,\, V)$ associated to it that keeps the Riemann surface fixed.

{}From the construction of $\text{At}(V)$ in \eqref{e12} it follows immediately that the dual vector
bundle $\text{At}(V)^*$ is a quotient of $\text{Diff}^1_X(V,\, V)^*\,=\, J^1(V)\otimes V^*$ (see \eqref{e13}).
To describe this quotient, consider the subbundle $V\otimes K_X\, \subset\, J^1(V)$ in \eqref{ej}. Using
it, we have
\begin{equation}\label{sb1}
\text{ad}(V)\otimes K_X\, \subset\, \text{End}(V)\otimes K_X \,=\, (V\otimes K_X)\otimes V^*
\, \subset\,J^1(V)\otimes V^*\, ,
\end{equation}
where $\text{ad}(V)$ as before is the sheaf of trace zero endomorphisms of $V$. From \eqref{e12} it is deduced that
\begin{equation}\label{e14}
\text{At}(V)^*\,=\, (J^1(V)\otimes V^*)\big{/}(\text{ad}(V)\otimes K_X)\,=\,
\text{Diff}^1_X(V,\, V)^*\big{/}(\text{ad}(V)\otimes K_X)\, ,
\end{equation}
where the quotient is by the subbundle in \eqref{sb1}; note that $\text{Diff}^1_X(V,\, V)^*\,=\,
J^1(V)\otimes V^*$. In view of \eqref{e14}, by Serre duality,
\begin{equation}\label{e15}
H^1(X,\,\text{At}(V))^*\,=\, H^0
\left(X,\,\,\, \frac{J^1(V)\otimes V^*}{\text{ad}(V)\otimes K_X}\otimes K_X\right)\, .
\end{equation}

As in Section \ref{sec2.1}, let $K^{1/2}_X$ be a theta characteristic on $X$. Since the collection of
theta characteristics on $X$ is a discrete set, there is unique way to move the theta characteristic
when $X$ moves over a family of Riemann surfaces parametrized by a simply connected space. A consequence
of this observation will be explained now.

Consider the Atiyah exact sequence
$$
0\, \longrightarrow\, {\mathcal O}_X\, \longrightarrow\, \text{At}(K^{1/2}_X)
\, \longrightarrow\, TX \, \longrightarrow\, 0
$$
in \eqref{e11} for $V\,=\, K^{1/2}_X$. Let
\begin{gather}
0\,=\, H^0(X,\, TX) \, \longrightarrow\,
H^1(X,\, {\mathcal O}_X)\, \stackrel{h}{\longrightarrow}\nonumber\\
H^1(X,\, \text{At}(K^{1/2}_X))
\, \stackrel{\xi}{\longrightarrow}\, H^1(X,\, TX) \, \longrightarrow\, H^2(X,\, {\mathcal O}_X)\,=\, 0\label{e17}
\end{gather}
be the long exact sequence of cohomologies associated to it.
The above observation implies that there is a canonical homomorphism
\begin{equation}\label{e16}
\phi\, :\, H^1(X,\, TX) \, \longrightarrow\, H^1(X,\, \text{At}(K^{1/2}_X))
\end{equation}
that sends an infinitesimal deformation of $X$ to the corresponding infinitesimal deformation of
the pair $(X,\, K^{1/2}_X)$. In particular, we have $\xi\circ\phi\,=\, \text{Id}_{H^1(X,\, TX)}$, where
$\xi$ is the homomorphism in \eqref{e17}.

An alternative description of the homomorphism $\phi$ in \eqref{e16} is the following.
Let $\{U_i\}_{i\in I}$ be a covering of $X$ by open subsets, and let
$$
\lambda_{i,j}\, :\, U_{ij} \, =\,U_i\cap U_j \, \longrightarrow\, TU_{ij}\, ,\, \ \ i,\, j\, \in\, I,
$$
be a $1$-cocycle giving an element of $H^1(X,\, TX)$. Then $\lambda_{i,j}$ acts on $H^0(U_{ij},\,
K^{1/2}_X)$ by Lie derivation. To define the operation Lie derivation, take a locally defined holomorphic vector
field $v$ and a locally defined holomorphic section $s$ of $K^{1/2}_X$; so $s\otimes s$ is a locally defined
holomorphic $1$-from on $X$. Now define $L_vs$ by the equation
\begin{equation}\label{eld}
s\otimes L_v s \,=\, \frac{1}{2} L_v (s\otimes s)\,=\, \frac{1}{2}\big(i_vd(s\otimes s)+ d((s\otimes s)(v))\big)\, .
\end{equation}
Using this action by Lie derivation, $\{\lambda_{i,j}\}$ is considered as a $1$-cocycle with values in
$\text{At}(K^{1/2}_X)$. The corresponding element of $H^1(X,\, \text{At}(K^{1/2}_X))$ is the image, under
the homomorphism $\phi$ in \eqref{e16}, of the cohomology class of $\{\lambda_{i,j}\}\,\in\, H^1(X,\, TX)$.

\begin{lemma}\label{lem2}
For any holomorphic vector bundle $V$ on $X$, there is a natural isomorphism
$$
\zeta\, :\, H^1(X,\, {\rm At}(V)) \,\stackrel{\sim}{\longrightarrow}\, H^1(X,\, {\rm At}(V\otimes K^{-1/2}_X))\, .
$$
\end{lemma}

\begin{proof}
Fix $X$, $V$ and a theta characteristic $K^{1/2}_X$ on $X$.
Take any infinitesimal deformation $v\, \in\, H^1(X,\, {\rm At}(V))$ of $(X,\, V)$. It corresponds to a family
of curves ${\mathcal X}\, \longrightarrow\, {\rm Spec}\, {\mathbb C}[\epsilon]/\epsilon^2$ together
with a holomorphic vector bundle ${\mathcal V}\, \longrightarrow\, {\mathcal X}$. The fiber of $\mathcal X$
over ${\rm Spec}\,\mathbb C$ is $X$. Let $v'\, \in\, H^1(X,\,TX)$ be the image of $v$ under the homomorphism
$H^1(X,\, {\rm At}(V))\, \longrightarrow\, H^1(X,\,TX)$ in \eqref{e11b}. The image
of $v'$ under the homomorphism $\phi$ in \eqref{e16} gives
a relative theta line bundle ${\mathcal K}^{1/2}\, \longrightarrow\, {\mathcal X}$, whose
restriction to $X$ is the chosen theta characteristic $K^{1/2}_X$. Now the
pair $({\mathcal X},\, {\mathcal V}\otimes ({\mathcal K}^{1/2})^*)$ corresponds to an element
$v''\, \in\, H^1(X,\, {\rm At}(V\otimes K^{-1/2}_X))$. The map
$$
\zeta\,\, :\,\, H^1(X,\, {\rm At}(V)) \,\longrightarrow\, H^1(X,\, {\rm At}(V\otimes K^{-1/2}_X))\, ,
\ \ v\, \longmapsto\, v''
$$
constructed above is evidently an isomorphism.

The above homomorphism $\zeta$ can be described in terms of the cocycles as follows.
Let $\{U_i\}_{i\in I}$ be a covering of $X$ by open subsets, and let
$$
\delta_{i,j}\, :\, U_{ij} \, =\,U_i\cap U_j \, \longrightarrow\, {\rm At}(V)\big\vert_{U_{ij}}
\, ,\, \ \ i,\, j\, \in\, I,
$$
be a $1$-cocycle giving an element $c\, \in\, H^1(X,\, {\rm At}(V))$. So
$$
\sigma\circ\delta_{i,j}\, :\, U_{ij} \, =\,U_i\cap U_j \, \longrightarrow\, TU_{ij}\, ,\, \ \ i,\, j\, \in\, I,
$$
is a $1$-cocycle with values in $TX$, where $\sigma$ is the projection in \eqref{e11}. Its cohomology
class in $H^1(X,\, TX)$ coincides with the image of $c$ by the homomorphism in \eqref{e11b}.
For any holomorphic sections
$$
v\,\in\, H^0\left(U_{ij},\, V\big\vert_{U_{ij}}\right),\ \ \kappa\, \in\, H^0\left(U_{ij},\,
K^{-1/2}_X\big\vert_{U_{ij}}\right)\, ,
$$
define $\Delta_{ij}(v\otimes \kappa)\,=\, \delta_{i,j}(v)\otimes\kappa + v\otimes L_{\sigma\circ\delta_{i,j}}\kappa$,
where $L_{\sigma\circ\delta_{i,j}}$ is a the Lie derivation by the vector field
$\sigma\circ\delta_{i,j}$ (Lie derivation is defined in \eqref{eld}); recall that $\delta_{i,j}$ is a first order holomorphic differential
operator $V\big\vert_{U_{ij}}\, \longrightarrow\, V\big\vert_{U_{ij}}$. It is straightforward to check that
$$
\Delta_{ij}((fv)\otimes \kappa)\,=\,\Delta_{ij}(v\otimes (f\kappa))
$$
for any holomorphic function $f$ on $U_{ij}$. Consequently, we have
$$
\Delta_{ij}\,\,\in\,\, H^0(U_{ij},\, {\rm At}(V\otimes K^{-1/2}_X))\, .
$$
The homomorphism $\zeta$ takes the cohomology class $c$ of $\{\delta_{i,j}\}$ to the
cohomology class of $\{\Delta_{ij}\}$ in $H^1(X,\, {\rm At}(V\otimes K^{-1/2}_X))$.
\end{proof}

Now take a triple $(X,\, V,\, D)$, where $X$ is a compact Riemann surface of
genus $g$, $V$ is a holomorphic vector bundle on $X$, and $D$ is a holomorphic
connection on $V$. We recall that holomorphic connections on $V$ are precisely the holomorphic splittings
of the Atiyah exact sequence in \eqref{e11} \cite{At}.
The connection $D$ on $V$ produces a holomorphic differential operator of order one
\begin{equation}\label{f1}
\widetilde{D}\,\, :\,\, \text{At}(V)\, \longrightarrow\,\text{End}(V)\otimes K_X
\end{equation}
which is constructed as follows. Let $p_1\,:\, \text{At}(V)\, \longrightarrow\, \text{End}(V)$
be the holomorphic projection defined by the holomorphic splitting of \eqref{e11} given by $D$. Let
$$
D'\, :\, \text{End}(V) \, \longrightarrow\, \text{End}(V)\otimes K_X
$$
be the holomorphic connection on $\text{End}(V)$ induced by the connection $D$ on $V$. Then $$\widetilde{D}\,=\,
D'\circ p_1\, .$$ Note that $\widetilde{D}$ is not ${\mathcal O}_X$--linear; it is a differential operator of
order one.

Let
\begin{equation}\label{f2}
{\mathcal A}_\bullet\,\, :\,\,{\mathcal A}_0\,:=\,\text{At}(V)\,\stackrel{\widetilde{D}}{\longrightarrow}
\, {\mathcal A}_1\,:=\, \text{End}(V)\otimes K_X
\end{equation}
be the two term complex of sheaves on $X$, where $\widetilde{D}$ is the differential
operator in \eqref{f1}, and ${\mathcal A}_i$ is at the $i$-th position. The infinitesimal deformations
of the triple $(X,\, V,\, D)$ are parametrized by the first hypercohomology
\begin{equation}\label{eh}
{\mathbb H}^1({\mathcal A}_\bullet)\, ,
\end{equation}
where ${\mathcal A}_\bullet$ is the complex constructed in \eqref{f2} \cite[p.~ 1415, Proposition 4.4]{Ch}
(see also \cite{BHH}).

Consider the homomorphisms of complexes
$$
\begin{matrix}
&& 0 & \longrightarrow & \text{End}(V)\otimes K_X\\
&& \Big\downarrow &&\,\,\,\,\,\, \Big\downarrow{\rm id}\\
{\mathcal A}_\bullet & : & \text{At}(V)& \stackrel{\widetilde{D}}{\longrightarrow} &
\text{End}(V)\otimes K_X\\
&&\,\,\,\,\,\, \Big\downarrow{\rm id} && \Big\downarrow\\
&& \text{At}(V) & \longrightarrow & 0
\end{matrix}
$$
It induces homomorphisms
$${\mathbb H}^1({\mathcal A}_\bullet)\,\longrightarrow\, H^1(X,\, {\rm At}(V))\ \ \text{ and }\ \
H^0(X,\, \text{End}(V)\otimes K_X)\,\longrightarrow\,{\mathbb H}^1({\mathcal A}_\bullet)\,.
$$
The first homomorphism corresponds to the forgetful map that sends an infinitesimal deformation of the triple
$(X,\, V,\, D)$ to the infinitesimal deformation of the pair $(X,\, V)$ obtained from it by simply forgetting
the connection. The second homomorphism corresponds to the map that sends an infinitesimal deformation of the
connection $D$ to the infinitesimal deformation of the triple $(X,\, V,\, D)$ produced by it by keeping
the pair $(X,\, V)$ fixed.

\section{Affine space structure}\label{sec5}

As before, $F$ is a stable vector bundle on $X$ of rank $r$ and degree zero.
Set $V\,=\, F\otimes K^{-1/2}_X$, and consider the vector space
$$
H^0\left(X,\, \,\,\frac{J^1(F\otimes K^{-1/2}_X)\otimes F^*\otimes K^{1/2}_X}{\text{ad}
(F\otimes K^{-1/2}_X)\otimes K_X} \otimes K_X\right)
$$
constructed in \eqref{e15}. Recall the quotient space ${\mathcal D}(F)$ in \eqref{qe}.

\begin{proposition}\label{prop2}
The space ${\mathcal D}(F)$ in \eqref{qe} is an affine space modeled on the above vector space
\begin{equation}\label{e18}
H^0\left(X,\,\,\, \frac{J^1(F\otimes K^{-1/2}_X)\otimes F^*\otimes K^{1/2}_X}{{\rm ad}(F\otimes K^{-1/2}_X)
\otimes K_X}\otimes K_X\right)\, .
\end{equation}
\end{proposition}

\begin{proof}
Using the natural inclusion map $\iota$ in \eqref{e8b}, we have
\begin{gather}
\text{ad}(F\otimes K^{-1/2}_X)\otimes K^{\otimes 2}_X\,\hookrightarrow\,\text{End}
(F\otimes K^{-1/2}_X)\otimes K^{\otimes 2}_X\,=\,\text{End}(F)\otimes K^{\otimes 2}_X\nonumber\\
=\, (F\otimes K^{1/2}_X)\otimes (F^*\otimes K^{3/2}_X)\,
\xrightarrow{\iota\otimes {\rm Id}_{F^*\otimes K^{3/2}_X}}\,
J^1(F\otimes K^{-1/2}_X)\otimes F^*\otimes K^{3/2}_X\nonumber
\end{gather}
(see also \eqref{e10}). It produces a short exact sequence of vector bundles on $X$
\begin{equation}\label{j1}
0\, \longrightarrow\, \text{ad}(F\otimes K^{-1/2}_X)\otimes K^{\otimes 2}_X\, \longrightarrow\, 
J^1(F\otimes K^{-1/2}_X)\otimes F^*\otimes K^{3/2}_X
\end{equation}
$$
\longrightarrow\, 
\frac{J^1(F\otimes K^{-1/2}_X)\otimes F^*\otimes K^{1/2}_X}{\text{ad}(F\otimes K^{-1/2}_X)
\otimes K_X}\otimes K_X \, \longrightarrow\, 0\, .
$$
Since $F$ is stable, we have $$H^1(X,\, \text{ad}(F\otimes K^{-1/2}_X)\otimes K^{\otimes 2}_X)\,=\,
H^0(X,\, \text{ad}(F)\otimes TX)^*\,=\, 0$$
because $\text{ad}(F)\otimes TX$ is polystable of negative degree (recall that $g\, \geq\, 2$). Therefore,
the long exact sequence of cohomologies associated to this short exact sequence in \eqref{j1}
gives a short exact sequence 
\begin{equation}\label{e19}
0\, \longrightarrow\, H^0(X,\, \text{ad}(F\otimes K^{-1/2}_X)\otimes K^{\otimes 2}_X)\, \longrightarrow\,
H^0(X,\, J^1(F\otimes K^{-1/2}_X)\otimes F^*\otimes K^{3/2}_X)
\end{equation}
$$
\longrightarrow\,
H^0\left(X,\, \,\,\frac{J^1(F\otimes K^{-1/2}_X)\otimes F^*\otimes K^{1/2}_X}{\text{ad}
(F\otimes K^{-1/2}_X)\otimes K_X} \otimes K_X\right)
\, \longrightarrow\, 0\, .
$$
Consequently, the vector space in \eqref{e18} coincides with the quotient space
$$
{\mathcal Q}(F)\, =:\,
\frac{H^0(X,\, J^1(F\otimes K^{-1/2}_X)\otimes F^*\otimes K^{3/2}_X)}{H^0
(X,\, \text{ad}(F\otimes K^{-1/2}_X)\otimes K^{\otimes 2}_X)}
$$
\begin{equation}\label{e24}
=\,
\frac{H^0(X,\, J^1(F\otimes K^{-1/2}_X)\otimes F^*\otimes K^{3/2}_X)}{
H^0(X,\, \text{ad}(F)\otimes K^{\otimes 2}_X)}\, .
\end{equation}

Now, we have a natural inclusion map
$$
H^0(X,\, J^1(F\otimes K^{-1/2}_X)\otimes F^*\otimes K^{3/2}_X)\,=\,
H^0(X,\, {\rm Hom}(F\otimes K^{-/2}_X,\, J^1(F\otimes K^{-1/2}_X)\otimes K_X))
$$
\begin{equation}\label{l1}
\, \hookrightarrow\, H^0(X,\,{\rm Hom}(J^1(F\otimes K^{-1/2}_X),\, J^1(F\otimes K^{-1/2}_X)\otimes K_X))
\end{equation}
that sends any homomorphism $\eta\,:\,F\otimes K^{-1/2}_X\, \longrightarrow\,
J^1(F\otimes K^{-1/2}_X)\otimes K_X$ over $X$ to
$$
\eta\circ q_0\, :\, J^1(F\otimes K^{-1/2}_X) \, \longrightarrow\, J^1(F\otimes K^{-1/2}_X)\otimes K_X\, ,
$$
where $q_0$ is the
projection in \eqref{e8b}. Using the inclusion map in \eqref{l1}, the natural action of the vector space
$H^0(X,\,{\rm Hom}(J^1(F\otimes K^{-1/2}_X),\, J^1(F\otimes K^{-1/2}_X)\otimes K_X))$ on the
affine space ${\rm Conn}(J^1(F\otimes K^{-1/2}_X))$ (defined in \eqref{c0}) restricts to an action of
$H^0(X,\, J^1(F\otimes K^{-1/2}_X)\otimes F^*\otimes K^{3/2}_X)$ on ${\rm Conn}(J^1(F\otimes K^{-1/2}_X))$.

Recall that Corollary \ref{cor1} says that ${\rm Conn}(J^1(F\otimes K^{-1/2}_X))\,=\,
\widetilde{\mathcal C}(F)$. Therefore, the vector space $H^0(X,\, J^1(F\otimes K^{-1/2}_X)
\otimes F^*\otimes K^{3/2}_X)$ acts on $\widetilde{\mathcal C}(F)$.

It can be shown that the above action of $H^0(X,\, J^1(F\otimes K^{-1/2}_X)
\otimes F^*\otimes K^{3/2}_X)$ on $\widetilde{\mathcal C}(F)$ produces an action of $H^0(X,\,
J^1(F\otimes K^{-1/2}_X)\otimes F^*\otimes K^{3/2}_X)$ on the quotient space ${\mathcal C}(F)$
in Definition \ref{def3}. Indeed, the action of $H^0(X,\, J^1(F\otimes K^{-1/2}_X)
\otimes F^*\otimes K^{3/2}_X)$ on $\widetilde{\mathcal C}(F)$, and also the quotient map
$\widetilde{\mathcal C}(F)\, \longrightarrow\, {\mathcal C}(F)$, are both constructed using the
action of $H^0(X,\, {\rm End}(J^1(F\otimes K^{-1/2}_X))\otimes K_X)$ on $\widetilde{\mathcal C}(F)$; as
$S\,=\, F\otimes K^{1/2}_X$ (see \eqref{e8}), we have $\text{End}(S)\,=\, \text{End}(F)\,=\,
\text{End}(F\otimes K^{-1/2}_X)$. Since the group
$H^0(X,\, {\rm End}(J^1(F\otimes K^{-1/2}_X))\otimes K_X)$ is abelian, the two
actions on $\widetilde{\mathcal C}(F)$ commute, and hence we get an action of $H^0(X,\,
J^1(F\otimes K^{-1/2}_X)\otimes F^*\otimes K^{3/2}_X)$ on ${\mathcal C}(F)$.

For the above action of $H^0(X,\, J^1(F\otimes K^{-1/2}_X)\otimes F^*\otimes K^{3/2}_X)$ on
${\mathcal C}(F)$, it is evident that the subspace
$$
H^0(X,\, \text{ad}(F)\otimes K^{\otimes 2}_X)\,=\,H^0(X,\, \text{ad}(F\otimes K^{-1/2}_X)\otimes K^{\otimes 2}_X)
$$
$$
\, \subset\,
H^0(X,\, J^1(F\otimes K^{-1/2}_X)\otimes F^*\otimes K^{3/2}_X)
$$
in \eqref{e19} acts trivially. Indeed, this follows from the fact that ${\mathcal C}(F)$ is the
quotient of $\widetilde{\mathcal C}(F)$ by the action of $H^0(X,\, \text{ad}(F)\otimes K^{\otimes 2}_X)$.
Thus we have an action on ${\mathcal C}(F)$ of the quotient space
${\mathcal Q}(F)$ in \eqref{e24}. This action of ${\mathcal Q}(F)$ on ${\mathcal C}(F)$
is free, because
\begin{itemize}
\item the action of $H^0(X,\,{\rm Hom}(J^1(F\otimes K^{-1/2}_X),\, J^1(F\otimes K^{-1/2}_X)\otimes K_X))$ on
${\rm Conn}(J^1(F\otimes K^{-1/2}_X))$ is free, and

\item ${\mathcal Q}(F)$ and ${\mathcal C}(F)$ are constructed by quotienting with the same subspace, namely
$$H^0(X,\, \text{ad}(F)\otimes K^{\otimes 2}_X)
\,\subset\, H^0(X,\, J^1(F\otimes K^{-1/2}_X)\otimes F^*\otimes K^{3/2}_X)$$
(see \eqref{e24} and Definition \ref{def3}).
\end{itemize}
Finally from the
structure of $\text{Aut}(J^1(F\otimes K^{-1/2}_X))$ shown in \eqref{aj} it follows that the
above action of ${\mathcal Q}(F)$ on ${\mathcal C}(F)$ produces an action of
${\mathcal Q}(F)$ on the quotient space
${\mathcal D}(F)$ in \eqref{qe}.

Using Lemma \ref{lem5} it follows that the above action of
${\mathcal Q}(F)$ on ${\mathcal D}(F)$ is transitive. This action is
also free. Indeed, this follows from \eqref{aj} and the above observation that the
action of ${\mathcal Q}(F)$ on ${\mathcal C}(F)$
is free.

We noted earlier that ${\mathcal Q}(F)$ is identified with the vector space in \eqref{e18}. Therefore,
the above action of ${\mathcal Q}(F)$ on ${\mathcal D}(F)$ produces an action, on ${\mathcal D}(F)$,
of the vector space in \eqref{e18}. Now the proposition follows from the above properties of the
action of ${\mathcal Q}(F)$ on ${\mathcal D}(F)$.
\end{proof}

Proposition \ref{prop2} and Lemma \ref{lem2} together give the following:

\begin{corollary}\label{cor4}
The space ${\mathcal D}(F)$ is an affine space modeled on the complex vector space $H^1(X,\, {\rm At}(F))^*$.
\end{corollary}

\begin{proof}
{}From Lemma \ref{lem2} we know that $H^1(X,\, {\rm At}(F))^*\,=\, H^1(X,\, {\rm At}(F\otimes
K^{-1/2}_X))^*$. From \eqref{e15} it follows that
$$
H^1(X,\, {\rm At}(F\otimes K^{-1/2}_X))^*
$$
$$
\,=\, 
H^0\left(X,\, \left((J^1(F\otimes K^{-1/2}_X)\otimes F^*\otimes K^{1/2}_X)\big{/}(\text{ad}(F\otimes K^{-1/2}_X)
\otimes K_X)\right)\otimes K_X\right)\, .
$$
Now Proposition \ref{prop2} completes the proof.
\end{proof}

Let ${\mathcal M}^\theta_g$ denote the moduli space of irreducible smooth complex projective curves of genus $g$,
with $g\,\geq\, 2$, equipped with a theta characteristic.
It is a smooth orbifold of complex dimension $3g-3$. We note that ${\mathcal M}^\theta_g$ is not connected;
the loci of curves with an odd theta characteristic and curves with an even theta characteristic are
disconnected. For any fixed $r\, \geq\, 1$, let
\begin{equation}\label{e25}
\beta\, :\, {\mathcal B}_g(r)\, \longrightarrow\, {\mathcal M}^\theta_g
\end{equation}
be the moduli space of triples of the form $(X,\, K^{1/2}_X,\, F)$, where
\begin{itemize}
\item $X$ is a compact connected Riemann surface of genus $g$,

\item $K^{1/2}_X$ is a theta characteristic on $X$, and

\item $F$ is a stable vector bundle on $X$ of rank $r$ and degree zero.
\end{itemize}
The map $\beta$ in \eqref{e25} sends any $(X,\, K^{1/2}_X,\, F)$ to $(X,\, K^{1/2}_X)$
by forgetting $F$. The moduli space ${\mathcal B}_g(r)$ is a smooth orbifold of
complex dimension $3g+r^2(g-1)-2$. Note that ${\mathcal B}_g(r)$ is not connected
as ${\mathcal M}^\theta_g$ is not connected. Let
\begin{equation}\label{e26}
\gamma\, :\, {\mathcal H}_g(r)\, \longrightarrow\, {\mathcal B}_g(r)
\end{equation}
be the moduli space of quadruples of the form $(X,\, K^{1/2}_X,\, F,\, D)$, where
\begin{itemize}
\item $X$ is a compact connected Riemann surface of genus $g$,

\item $K^{1/2}_X$ is a theta characteristic on $X$,

\item $F$ is a stable vector bundle on $X$ of rank $r$ and degree zero, and

\item $D\, \in\,{\mathcal D}(F)$ (see \eqref{qe}).
\end{itemize}
The projection $\gamma$ in \eqref{e26} sends any $(X,\, K^{1/2}_X,\, F,\, D)$ to $(X,\, K^{1/2}_X,\, F)$
by simply forgetting $D$. The moduli space ${\mathcal H}_g(r)$ is a smooth orbifold of complex dimension
$2(3g+r^2(g-1)-2)$.

\begin{theorem}\label{thm1}
The algebraic fiber bundle $\gamma\, :\, {\mathcal H}_g(r)\, \longrightarrow\, {\mathcal B}_g(r)$
in \eqref{e26} is an algebraic affine bundle modeled on the holomorphic cotangent bundle
$T^*{\mathcal B}_g(r)$. In other words, ${\mathcal H}_g(r)$ is an algebraic torsor over
${\mathcal B}_g(r)$ for $T^*{\mathcal B}(r)$.
\end{theorem}

\begin{proof}
{}From Proposition \ref{prop1} and Corollary \ref{cor1} we know that $\gamma$ is surjective.

Since the space of infinitesimal deformations of a pair $(X,\, K^{1/2}_X,\, F)\, \in\, {\mathcal B}_g(r)$ are
parametrized by $H^1(X,\, \text{At}(F))$ (see \eqref{e11b} and \eqref{e16}), we conclude that
$$
T^*_{(X,K^{1/2}_X,F)}{\mathcal B}_g(r)\,=\, H^1(X,\, \text{At}(F))^*\, .
$$
Therefore, the theorem follows from Corollary \ref{cor4}. It should be clarified that the canonical
nature of the action of $H^1(X,\, {\rm At}(F))^*$ on ${\mathcal D}(F)$ in Corollary \ref{cor4} ensures
that it extends to any given family of Riemann surfaces equipped with a theta characteristic and a
stable vector bundle of rank $r$ and degree zero. 
\end{proof}

On ${\mathcal B}_g(r)$ there is a natural reduced divisor $\Theta$ which is defined as follows:
\begin{equation}\label{e27}
\Theta\, :=\, \{(X,\, K^{1/2}_X,\, F)\,\in\, {\mathcal B}_g(r)\ \big\vert\ H^0(X,\, F\otimes K^{1/2}_X)\,
\not=\, 0\}\, \subset\, {\mathcal B}_g(r)\,.
\end{equation}
Since $\chi(X,\, F\otimes K^{1/2}_X)\,=\, 0$ (Riemann--Roch theorem), it follows that $H^0(X,\, F\otimes K^{1/2}_X)
\,\not=\, 0$ if and only if $H^1(X,\, F\otimes K^{1/2}_X)\,\not=\, 0$. For any $(X,\, K^{1/2}_X,\, F)\,\in\,\Theta$,
we have
\begin{equation}\label{o-1}
(X,\, K^{1/2}_X,\, F^*)\,\in\,\Theta\, ,
\end{equation}
because $H^0(X,\, F^*\otimes K^{1/2}_X)\,=\, H^1(X,\, F\otimes K^{1/2}_X)^*\,\not=\, 0$ by Serre duality.

\section{Differential operators, integral kernels and {\it r}--opers} 

\subsection{Differential operators and {\it r}--opers}

As before, take any $(X,\, K^{1/2}_X,\, F)\,\in\, {\mathcal B}_g(r)$ (see \eqref{e25}).
Let $D$ be a holomorphic connection on $E\, :=\, J^1(F\otimes K^{-1/2}_X)$; it
exists by Proposition \ref{prop1}. We know that the triple
$(E,\, F\otimes K^{1/2}_X,\, D)$ is a $r$--oper (see Corollary \ref{cor1} and Proposition \ref{prop1}),
where $F\otimes K^{1/2}_X$ is considered as a subbundle of $E$ using
$\iota$ in \eqref{e8b}. Consider the automorphism
\begin{equation}\label{l2}
\varphi\, :\, E\,=\, J^1(F\otimes K^{-1/2}_X)\,\stackrel{\sim}{\longrightarrow}\, J^1(F\otimes K^{-1/2}_X)\,=\, E
\end{equation}
constructed in \eqref{e3}. It should be clarified that the construction of $\varphi$ uses $D$
in an essential way; so $\varphi$ need not be the identity map of $E$. Imitating the construction of
$\varphi$ in \eqref{e3} we will construct a homomorphism
\begin{equation}\label{vp2}
\varphi_2\, :\, E\,=\, J^1(F\otimes K^{-1/2}_X)\, \longrightarrow\, J^2(F\otimes K^{-1/2}_X)\, .
\end{equation}
To construct $\varphi_2$,
take any $x\, \in\, X$ and any $v\, \in\, E_x$. As in the proof of Lemma \ref{lem1}, let
$\widehat{v}$ be the unique flat section of $E$ for the
connection $D$, defined on a simply connected open neighborhood $U\, \subset\, X$ of $x$, such that $\widehat{v}(x)\,
=\,v$. So $q_0(\widehat{v})\, \in\, H^0(U,\, F\otimes K^{-1/2}_X)$, where
$q_0$ is the projection in \eqref{e8b}. Now restrict $q_0(\widehat{v})$ to the second order infinitesimal
neighborhood of $x$; let $\widehat{v}'_2\, \in\, J^2(F\otimes K^{-1/2}_X)_x$ be the element obtained this
way from $q_0(\widehat{v})$. The homomorphism $\varphi_2$ sends any $v\, \in\, E_x$, $x\, \in\, X$,
to $\widehat{v}'_2\, \in\, J^2(F\otimes K^{-1/2}_X)_x$ constructed above from it.

Let
\begin{equation}\label{2j}
0\, \longrightarrow\, F\otimes K^{-1/2}_X\otimes K^{2}_X\,=\,F\otimes K^{3/2}_X \, \xrightarrow{\iota_2}\,
J^2(F\otimes K^{-1/2}_X)\, \xrightarrow{q_2}\, J^1(F\otimes K^{-1/2}_X)\, \longrightarrow\, 0
\end{equation}
be the canonical short exact sequence of jet bundles. The homomorphism
$$
\varphi_2\circ\varphi^{-1}\, :\, J^1(F\otimes K^{-1/2}_X)\, \longrightarrow\, J^2(F\otimes K^{-1/2}_X)\, ,
$$
where $\varphi_2$ and $\varphi$ are the homomorphisms in \eqref{vp2} and \eqref{l2} respectively, satisfies the equation
\begin{equation}\label{3j}
q_2\circ (\varphi_2\circ\varphi^{-1})\,=\, \text{Id}_{J^1(F\otimes K^{-1/2}_X)}\, ,
\end{equation}
where $q_2$ is the projection in \eqref{2j}.
In other words, $\varphi_2\circ\varphi^{-1}$ gives a holomorphic splitting of the short exact sequence in
\eqref{2j}. Therefore, there is a unique holomorphic homomorphism
\begin{equation}\label{4j}
\Delta_D\, :\, J^2(F\otimes K^{-1/2}_X) \, \longrightarrow\,F\otimes K^{3/2}_X
\end{equation}
such that
\begin{itemize}
\item $\text{kernel}(\Delta_D)\,=\, {\rm image}(\varphi_2\circ\varphi^{-1})\,=\, {\rm image}(\varphi_2)$, and

\item $\Delta_D\circ\iota_2\,=\, \text{Id}_{F\otimes K^{3/2}_X}$, where $\iota_2$ is the homomorphism in
\eqref{2j}.
\end{itemize}

We note that the homomorphism $\Delta_D$ in \eqref{4j} defines a holomorphic differential operator of
order two
\begin{equation}\label{5j}
\Delta_D\, \in\, H^0(X,\, \text{Diff}^2_X(F\otimes K^{-1/2}_X,\, F\otimes K^{3/2}_X))
\end{equation}
from $F\otimes K^{-1/2}_X$ to $F\otimes K^{3/2}_X$. The symbol of any
holomorphic differential operator of order two from $F\otimes K^{-1/2}_X$ to $F\otimes K^{3/2}_X$
is a holomorphic section of
$$
\text{Hom}(F\otimes K^{-1/2}_X,\, F\otimes K^{3/2}_X)\otimes (TX)^{\otimes 2}\,=\,
\text{End}(F)\, .
$$
{}From the above equation $\Delta_D\circ\iota_2\,=\, \text{Id}_{F\otimes K^{3/2}_X}$ it follows immediately
that the symbol of the differential operator $\Delta_D$ in \eqref{5j} is actually $\text{Id}_F\, \in\,
H^0(X,\, \text{End}(F))$.

For any $A\, \in\, H^0(X,\, \text{Diff}^2_X(F\otimes K^{-1/2}_X,\, F\otimes K^{3/2}_X))$ and
\begin{gather}
B\, \in\, H^0(X,\, \text{Diff}^0_X(F\otimes K^{-1/2}_X,\, F\otimes K^{3/2}_X))\nonumber\\
=\, H^0(X,\, \text{Hom}(F\otimes K^{-1/2}_X,\, F\otimes K^{3/2}_X))
\,=\, H^0(X,\, \text{End}(F)\otimes K^2_X)\, ,\nonumber
\end{gather}
we have
$$
A+B\, \in\, H^0(X,\, \text{Diff}^2_X(F\otimes K^{-1/2}_X,\, F\otimes K^{3/2}_X))\, .
$$
Also, the symbol of $A+B$ coincides with the symbol of $A$, because $B$ is a lower order differential operator.

Two holomorphic differential operators $A_1,\, A_2\, \in\, H^0(X,\, \text{Diff}^2_X(F\otimes K^{-1/2}_X,\,
F\otimes K^{3/2}_X))$ will be called \textit{equivalent} if
\begin{equation}\label{de}
A_1-A_2\, \in\, H^0(X,\, \text{ad}(F)\otimes K^{2}_X)\, \subset\, H^0(X,\, \text{End}(F)\otimes K^2_X)\, .
\end{equation}

\begin{definition}\label{ds}
The space of all equivalence classes of differential operators
$$A\, \in\, H^0(X,\, \text{Diff}^2_X(F\otimes K^{-1/2}_X,\, F\otimes K^{3/2}_X))$$
such that the symbol of $A$ is ${\rm Id}_F\, \in\, H^0(X,\, \text{End}(F))$ will
be denoted by $\widetilde{\mathcal D}(X,\, K^{1/2}_X,\, F)$.
\end{definition}

The above construction of $\Delta_D$ in \eqref{5j} from $D$ gives the following:

\begin{lemma}\label{lem6}
For any $(X,\, K^{1/2}_X,\, F)\,\in\, {\mathcal B}_g(r)$, there is a natural map
$$
\Psi\, :\, {\mathcal D}(F)\, \longrightarrow\, \widetilde{\mathcal D}(X,\, K^{1/2}_X,\, F)\, ,
$$
where ${\mathcal D}(F)$ and $\widetilde{\mathcal D}(X,\, K^{1/2}_X,\, F)$ are constructed
in \eqref{qe} and Definition \ref{ds} respectively.
\end{lemma}

\begin{proof}
Let $\widetilde{\mathcal D}'(X,\, K^{1/2}_X,\, F)$ denote the space of all holomorphic
differential operators
$A\, \in\, H^0(X,\, \text{Diff}^2_X(F\otimes K^{-1/2}_X,\, F\otimes K^{3/2}_X))$ such that the
symbol of $A$ is ${\rm Id}_F\, \in\, H^0(X,\, \text{End}(F))$.
The construction of $\Delta_D$ in \eqref{5j} from $D$ clearly gives a map
$$
\Psi'\, :\, \widetilde{\mathcal{\mathcal C}}(F)\, \longrightarrow\, \widetilde{\mathcal D}'(X,\, K^{1/2}_X,\, F)
$$
(see Definition \ref{def2}). Now consider the quotient space ${\mathcal C}(F)$ of $\widetilde{\mathcal{\mathcal C}}(F)$
in Definition \ref{def3}. Take two element $D_1,\, D_2\, \in\, \widetilde{\mathcal{\mathcal C}}(F)$ that give the
same element of ${\mathcal C}(F)$. It is straightforward to check that the differential operators $\Psi'(D_1)$ and
$\Psi'(D_2)$ are equivalent. Therefore, $\Psi'$ gives a map
$$
\Psi''\, :\, {\mathcal C}(F)\, \longrightarrow\, \widetilde{\mathcal D}(X,\, K^{1/2}_X,\, F)\, .
$$
The map $\Psi''$ clearly factors through the quotient ${\mathcal D}(F)$ of ${\mathcal C}(F)$ in \eqref{qe}.
Hence $\Psi''$ produces a map $\Psi$ as in the statement of the lemma.
\end{proof}

We will construct an inverse of the map $\Psi$ in Lemma \ref{lem6}.

For a holomorphic vector bundle $W$ on $X$, there is a canonical commutative diagram of homomorphisms
\begin{equation}\label{6j}
\begin{matrix}
&& 0 && 0 &&\\
&& \Big\downarrow && \Big\downarrow &&\\
0 &\longrightarrow & W\otimes K^2_X &\stackrel{\iota_2}{\longrightarrow} & J^2(W)
&\stackrel{q_2}{\longrightarrow} & J^1(W) &\longrightarrow & 0\\
&& \Big\downarrow && \,\,\, \Big\downarrow \mathbf{b} && \Big\Vert \\
0 &\longrightarrow & J^1(W)\otimes K_X & \stackrel{\iota'}{\longrightarrow} & J^1(J^1(W))
&\stackrel{q'}{\longrightarrow} & J^1(W) &\longrightarrow & 0\\
&&\Big\downarrow && \Big\downarrow &&\\
&& W\otimes K_X & = & W\otimes K_X &&\\
&& \Big\downarrow && \Big\downarrow && \\
&& 0 && 0 &&
\end{matrix}
\end{equation}
where the rows and columns are exact; the top (respectively, bottom) row is the one as in \eqref{2j}
(respectively, \eqref{ej}), while the left column is \eqref{ej} tensored with $K_X$. The map $\mathbf b$
is tautological; it follows from the definition of jet bundles. Now set
$$
W\,=\, F\otimes K^{-1/2}_X
$$
in \eqref{6j}, where $F$ is a stable vector bundle of rank $r$ and degree zero on $X$, and
$K^{1/2}_X$ is a theta characteristic on $X$. Let
$$
\Delta_0\, :\, J^2(F\otimes K^{-1/2}_X)\, \longrightarrow\, F\otimes K^{-1/2}_X\otimes K^2_X\,=\,
F\otimes K^{3/2}_X
$$
be a holomorphic homomorphism such that
\begin{equation}\label{7j}
\Delta_0\circ\iota_2\,=\,\text{Id}_{F\otimes K^{3/2}_X}\, ,
\end{equation}
where $\iota_2$ is the homomorphism in \eqref{6j}. In other words,
$$
\Delta_0\,\in\, H^0(X,\, \text{Diff}^2_X(F\otimes K^{-1/2}_X,\, F\otimes K^{3/2}_X))
$$
and the symbol of $\Delta_0$ is $\text{Id}_F\, \in\,
H^0(X,\, \text{End}(F))$ (it is equivalent to the equation in
\eqref{7j}). From \eqref{7j} it follows that $\Delta_0$ gives a holomorphic splitting of the
top row in \eqref{6j}. Consequently, there is a unique holomorphic homomorphism
$$
D'\, :\, J^1(F\otimes K^{-1/2}_X)\, \longrightarrow\, J^2(F\otimes K^{-1/2}_X)
$$
such that $\Delta_0\circ D'\,=\, 0$ and
\begin{equation}\label{8j}
q_2\circ D'\,=\, \text{Id}_{J^1(F\otimes K^{-1/2}_X)}\,,
\end{equation}
where $q_2$ is the projection in \eqref{6j}.

Now consider the composition of homomorphisms
\begin{equation}\label{9j}
D\, :=\, \mathbf{b}\circ D'\, :\, J^1(F\otimes K^{-1/2}_X)\, \longrightarrow\, J^1(J^1(F\otimes K^{-1/2}_X))\, ,
\end{equation}
where $\mathbf{b}$ is the homomorphism in \eqref{6j}. Since the diagram in \eqref{6j} is
commutative, from \eqref{8j} it follows that
$$
q'\circ D\,=\, q' \circ\mathbf{b}\circ D' \,=\,{\rm Id}_{J^1(F\otimes K^{-1/2}_X)}\circ q_2 \circ D'\,=\,
\text{Id}_{J^1(F\otimes K^{-1/2}_X)}\, ,
$$
where $q'$ is the projection in \eqref{6j}. In other words, $D$ in \eqref{9j} gives
a holomorphic connection on the vector bundle $J^1(F\otimes K^{-1/2}_X)$. The differential operator
$J^1(F\otimes K^{-1/2}_X)\, \longrightarrow\, J^1(F\otimes K^{-1/2}_X)\otimes K_X$
for the connection $D$ is the unique ${\mathcal O}_X$--linear homomorphism
$$
\widehat{D}\, :\, J^1(J^1(F\otimes K^{-1/2}_X))\, \longrightarrow\,J^1(F\otimes K^{-1/2}_X)\otimes K_X
$$
that satisfies the following conditions:
\begin{itemize}
\item $\text{kernel}(\widehat{D})\,=\, {\rm image}(D)$, and

\item $\widehat{D}\circ\iota'\,=\,
\text{Id}_{J^1(F\otimes K^{-1/2}_X)\otimes K_X}$, where $\iota'$ is the homomorphism in \eqref{6j}.
\end{itemize}

{}From Corollary \ref{cor1} we know that any holomorphic connection on $J^1(F\otimes K^{-1/2}_X)$
gives an element of ${\mathcal C}(F)$ (see Definition \ref{def3}). Hence the holomorphic connection $D$ in
\eqref{9j} gives an element
\begin{equation}\label{10j}
D_{\Delta_0}\, \in\, {\mathcal D}(F)\, ,
\end{equation}
where ${\mathcal D}(F)$ is the quotient of ${\mathcal C}(F)$ defined in \eqref{qe}.

\begin{lemma}\label{lem7}
For any $(X,\, K^{1/2}_X,\, F)\,\in\, {\mathcal B}_g(r)$ (see \eqref{e25}), there is a natural map
$$
\Phi\, :\, \widetilde{\mathcal D}(X,\, K^{1/2}_X,\, F)\, \longrightarrow\, {\mathcal D}(F)\, ,
$$
where $\widetilde{\mathcal D}(X,\, K^{1/2}_X,\, F)$ and ${\mathcal D}(F)$ are constructed
in Definition \ref{ds} and \eqref{qe} respectively.
\end{lemma}

\begin{proof}
The above construction of $D_{\Delta_0}$ (in \eqref{10j}) from $\Delta_0$ produces a map
\begin{equation}\label{pp}
\Phi'\, :\, \widetilde{\mathcal D}'(X,\, K^{1/2}_X,\, F)\, \longrightarrow\, {\mathcal D}(F),\,
\end{equation}
where $\widetilde{\mathcal D}'(X,\, K^{1/2}_X,\, F)$ is defined in the proof of Lemma \ref{lem6}.

Take $\Delta_1,\, \Delta_2\, \in\, \widetilde{\mathcal D}'(X,\, K^{1/2}_X,\, F)$ such that $\Delta_1$
is equivalent to $\Delta_2$ (see \eqref{de}). Let $D_1$ (respectively, $D_2$) be the holomorphic connection on
$J^1(F\otimes K^{-1/2}_X)$ corresponding to $\Delta_1$ (respectively, $\Delta_2$); see \eqref{9j}.
Since
$$
\Delta_1-\Delta_2\, \in\, H^0(X,\, \text{ad}(F)\otimes K^{2}_X)\, ,
$$
we have
$$
D_1-D_2\,=\, \Delta_1-\Delta_2\, \in\, H^0(X,\, \text{ad}(F)\otimes K^{2}_X)\, .
$$
Hence $D_1$ and $D_2$ give the same element of ${\mathcal C}(F)$ defined in Definition \ref{def3}.
This implies that $\Phi'$ in \eqref{pp} produces a map $\Phi$ as in the statement of the lemma.
\end{proof}

\begin{theorem}\label{thm2}
The two maps $\Psi$ and $\Phi$, constructed in Lemma \ref{lem6} and Lemma \ref{lem7} respectively,
are inverses of each other.
\end{theorem}

\begin{proof}
In view of the explicit nature of the maps $\Psi$ and $\Phi$, this is a verification by straightforward
computations. We omit the details.
\end{proof}

\subsection{Integral kernels and differential operators}

For $i\,=\, 1,\, 2$, let $$q_i\, :\, X\times X\, \longrightarrow\, X$$ be the projection to the $i$-th factor.
For holomorphic vector bundles $A,\, B$ on $X$, the holomorphic vector bundle $(q^*_1 A)\otimes (q^*_2 B)$ on
$X\times X$ will be denoted by $A\boxtimes B$. Let
$$
{\mathbf\Delta}\, :=\, \{(x,\, x)\,\,\mid\,\, x\in\, X\} \, \subset\, X\times X
$$
be the reduced diagonal divisor.

For holomorphic vector bundles $A,\, B$ on $X$, and a nonnegative integer $d$, we will construct a torsion sheaf
on $X\times X$ supported on the divisor $(d+1){\mathbf \Delta}\, \subset\, X\times X$. Let $r_A$ and
$r_B$ be the ranks of $A$ and $B$ respectively.

Consider the holomorphic vector bundles $$A\boxtimes (B^*\otimes K_X)\ \ \, \text{ and }\ \ \,
A\boxtimes (B^*\otimes K_X)\otimes {\mathcal O}_{X\times X}((d+1){\mathbf \Delta})$$ on $X\times X$. We note that
$A\boxtimes (B^*\otimes K_X)$ is a subsheaf of $A\boxtimes (B^*\otimes K_X)\otimes
{\mathcal O}_{X\times X}((d+1){\mathbf \Delta})$ because $(d+1){\mathbf \Delta}$ is an effective divisor on
$X\times X$. So we have a short exact sequence of coherent sheaves on $X\times X$
\begin{gather}
0\, \longrightarrow\, A\boxtimes (B^*\otimes K_X)\, \longrightarrow\,A\boxtimes (B^*\otimes K_X)\otimes
{\mathcal O}_{X\times X}((d+1){\mathbf \Delta})\nonumber\\
\longrightarrow\, {\mathcal Q}_d(A,\,B)\, :=\,
\frac{A\boxtimes (B^*\otimes K_X)\otimes
{\mathcal O}_{X\times X}((d+1){\mathbf \Delta})}{A\boxtimes (B^*\otimes K_X)} \, \longrightarrow\, 0\, ;
\label{k1}
\end{gather}
the support of the above quotient sheaf ${\mathcal Q}_d(A,\,B)$ in \eqref{k1}
is $(d+1){\mathbf \Delta}$. The direct image
\begin{equation}\label{k2}
{\mathcal K}_d(A,\,B)\, :=\, q_{1*} {\mathcal Q}_d(A,\,B)
\end{equation}
is a holomorphic vector bundle on $X$ of rank $r_Ar_B(d+1)$. It is known that
\begin{equation}\label{k3}
{\mathcal K}_d(A,\,B)\,=\, \text{Hom}(J^d(B),\, A)\,=\, \text{Diff}^d_X(B,\, A)\, ,
\end{equation}
where ${\mathcal K}_d(A,\,B)$ is constructed in \eqref{k2} (see \cite[Section 2.1]{BS}, \cite[p.~25, (5.1)]{Bi},
\cite[Seciton 3.1, p.~1314]{BB}).

For $d\, \geq\, 1$, the sheaf ${\mathcal Q}_d(A,\,B)$ in \eqref{k1} fits in the following short
exact sequence of sheaves on $X\times X$
\begin{gather}
0\, \longrightarrow\, {\mathcal Q}_{d-1}(A,\,B)\,:=\, \frac{A\boxtimes (B^*\otimes K_X)\otimes {\mathcal O}_{X\times X}
(d {\mathbf \Delta})}{A\boxtimes (B^*\otimes K_X)}\, \longrightarrow\, {\mathcal Q}_{d}(A,\,B)\label{ez1}\\
:=\, \frac{A\boxtimes (B^*\otimes K_X)\otimes
{\mathcal O}_{X\times X}((d+1){\mathbf \Delta})}{A\boxtimes (B^*\otimes K_X)} \, \longrightarrow\,
\frac{A\boxtimes (B^*\otimes K_X)\otimes
{\mathcal O}_{X\times X}((d+1){\mathbf \Delta})}{A\boxtimes (B^*\otimes K_X)\otimes {\mathcal O}_{X\times X}
(d {\mathbf \Delta})} \, \longrightarrow\, 0\nonumber .
\end{gather}
The above sheaf $\frac{A\boxtimes (B^*\otimes K_X)\otimes
{\mathcal O}_{X\times X}((d+1){\mathbf \Delta})}{A\boxtimes (B^*\otimes K_X)\otimes {\mathcal O}_{X\times X}
(d {\mathbf \Delta})}$ is supported on the reduced divisor ${\mathbf \Delta}$.
Taking direct image of the short exact sequence in \eqref{ez1} by the projection $q_1$ we get the following short exact sequence
holomorphic vector bundles on $X$
\begin{gather}
0\, \longrightarrow\, q_{1*} {\mathcal Q}_{d-1}(A,\,B)\,=\, {\mathcal K}_{d-1}(A,\,B) \, \longrightarrow\,
q_{1*} {\mathcal Q}_d(A,\,B)\,=\, {\mathcal K}_d(A,\,B)
\nonumber\\
\stackrel{p_0}{\longrightarrow}\, q_{1*} \left(\frac{A\boxtimes (B^*\otimes K_X)\otimes
{\mathcal O}_{X\times X}((d+1){\mathbf \Delta})}{A\boxtimes (B^*\otimes K_X)\otimes {\mathcal O}_{X\times X}
(d {\mathbf \Delta})}\right)\, \longrightarrow\, 0\, .\label{k4}
\end{gather}
Poincar\'e adjunction formula says that ${\mathcal O}_{X\times X}({\mathbf \Delta})\big\vert_{\mathbf \Delta}$
is the normal bundle of ${\mathbf \Delta}$ \cite[p.~146]{GH}.
So ${\mathcal O}_{X\times X}({\mathbf \Delta})\big\vert_{\mathbf \Delta}\,=\, TX$, using the identification of
${\mathbf \Delta}$ with $X$ defined by $x\, \longmapsto\, (x,\, x)$. Therefore, we have
$$
q_{1*} \left(\frac{A\boxtimes (B^*\otimes K_X)\otimes
{\mathcal O}_{X\times X}((d+1){\mathbf \Delta})}{A\boxtimes (B^*\otimes K_X)\otimes {\mathcal O}_{X\times X}
(d {\mathbf \Delta})}\right)\,=\, \text{Hom}(B,\, A)\otimes (TX)^{\otimes d}\, .
$$
Hence the isomorphism in \eqref{k3} and the projection $p_0$ in \eqref{k4} together produce a homomorphism
$$
\text{Diff}^d_X(B,\, A)\, \longrightarrow\, \text{Hom}(B,\, A)\otimes (TX)^{\otimes d}\, .
$$
The homomorphism of global sections corresponding to it
\begin{equation}\label{k5}
H^0(X,\, \text{Diff}^d_X(B,\, A))\, \longrightarrow\, H^0(X,\, \text{Hom}(B,\, A)\otimes (TX)^{\otimes d})
\end{equation}
is the symbol map on the global differential operators.

Let $F$ be a stable vector bundle on $X$ of rank $r$ and degree zero.
Using the isomorphism in \eqref{k3}, the space of holomorphic differential operators
$H^0(X,\, \text{Diff}^2_X(F\otimes K^{-1/2}_X,\, F\otimes K^{3/2}_X))$ has the following isomorphism:
\begin{equation}\label{di}
H^0(X,\, \text{Diff}^2_X(F\otimes K^{-1/2}_X,\, F\otimes K^{3/2}_X))
\end{equation}
$$
\stackrel{\sim}{\longrightarrow}\,
H^0\big(3\mathbf{\Delta},\, \big((F\otimes K^{3/2}_X)\boxtimes (F^*\otimes K^{3/2}_X)
\otimes {\mathcal O}_{X\times X}(3{\mathbf \Delta})\big)\big\vert_{3\mathbf{\Delta}}\big)\, .
$$
Now, for any
$$
s\, \in\, H^0\big(3\mathbf{\Delta},\, \big((F\otimes K^{3/2}_X)\boxtimes (F^*\otimes K^{3/2}_X)
\otimes{\mathcal O}_{X\times X}(3{\mathbf \Delta})\big)\big\vert_{3\mathbf{\Delta}}\big)\, ,
$$
let
\begin{equation}\label{s0}
s_0\, :=\, s\big\vert_{\mathbf{\Delta}}\, \in\,H^0\big(\mathbf{\Delta},\, (F\boxtimes F)\big\vert_{\mathbf{\Delta}}\big)
\,=\, H^0(X,\, \text{End}(F))
\end{equation}
be the restriction of it to $\mathbf{\Delta}\, \subset\, 3\mathbf{\Delta}$; note that
$((K^{3/2}_X\boxtimes K^{3/2}_X)\otimes {\mathcal O}_{X\times X}(3{\mathbf \Delta}))\big\vert_{\mathbf{\Delta}}$
is canonically trivialized, because ${\mathcal O}_{X\times X}({\mathbf \Delta})\big\vert_{\mathbf{\Delta}}
\,=\, TX$ (after identifying ${\mathbf \Delta}$ with $X$). Then $s_0$ coincides with
the symbol of the differential operator
$$
D_s\, \in\, H^0(X,\, \text{Diff}^2_X(F\otimes K^{-1/2}_X,\, F\otimes K^{3/2}_X))
$$
corresponding to the above section $s$ (see \eqref{k5} and \eqref{di}).

Consider the natural short exact sequence
$$
0\,\longrightarrow\, K^{\otimes 2}_{\mathbf{\Delta}}\, =\, K^{\otimes 2}_X\,\longrightarrow\,
{\mathcal O}_{3\mathbf{\Delta}}\,\longrightarrow\,{\mathcal O}_{2\mathbf{\Delta}}\,\longrightarrow\,0\, .
$$
Tensoring it with $(F\otimes K^{3/2}_X)\boxtimes (F^*\otimes K^{3/2}_X)\otimes
{\mathcal O}_{X\times X}(3{\mathbf \Delta})$, and then taking global sections, we see that the vector space
$$
H^0\big(\mathbf{\Delta},\, \big((F\otimes K^{3/2}_X)\boxtimes (F^*\otimes K^{3/2}_X)
\otimes{\mathcal O}_{X\times X}(3{\mathbf \Delta})\big)\big\vert_{\mathbf{\Delta}}\otimes
K^{\otimes 2}_{\mathbf{\Delta}}\big)\,=\,
H^0(X,\, \text{End}(F)\otimes K^{\otimes 2}_X)
$$
is a subspace of
$H^0\big(3\mathbf{\Delta},\, \big((F\otimes K^{3/2}_X)\boxtimes (F^*\otimes K^{3/2}_X)
\otimes{\mathcal O}_{X\times X}(3{\mathbf \Delta})\big)\big\vert_{3\mathbf{\Delta}}\big)$. 
Take two sections
$$
s,\, t\, \in\, H^0\big(3\mathbf{\Delta},\, \big((F\otimes K^{3/2}_X)\boxtimes (F^*\otimes K^{3/2}_X)
\otimes{\mathcal O}_{X\times X}(3{\mathbf \Delta})\big)\big\vert_{3\mathbf{\Delta}}\big)\, .
$$
Then the corresponding differential operators (see \eqref{di})
$$
D_s,\, D_t\, \in\, H^0(X,\, \text{Diff}^2_X(F\otimes K^{-1/2}_X,\, F\otimes K^{3/2}_X))
$$
are equivalent (see \eqref{de}) if and only if
$$
s-t\, \in\, H^0(X,\, \text{ad}(F)\otimes K^{\otimes 2}_X)\, \subset\, H^0(X,\, \text{End}(F)\otimes K^{\otimes 2}_X)\, .
$$
Consequently, we obtain the following description of $\widetilde{\mathcal D}(X,\, K^{1/2}_X,\, F)$ (see Definition \ref{ds})
in terms of integral kernels: $\widetilde{\mathcal D}(X,\, K^{1/2}_X,\, F)$ is identified with the
quotient of
$$
\{s\, \in\, H^0\big(3\mathbf{\Delta},\, \big((F\otimes K^{3/2}_X)\boxtimes (F^*\otimes K^{3/2}_X)
\otimes{\mathcal O}_{X\times X}(3{\mathbf \Delta})\big)\big\vert_{3\mathbf{\Delta}}\big)\,
\big\vert\, s\big\vert_{\mathbf{\Delta}}\,=\, H^0(X,\, \text{End}(F))\}
$$
by the subspace $H^0(X,\, \text{ad}(F)\otimes K^{\otimes 2}_X)$ of it.
Using this description of $\widetilde{\mathcal D}(X,\, K^{1/2}_X,\, F)$,
Theorem \ref{thm2} gives the following description of ${\mathcal H}_g(r)$ in terms
of the integral kernels.

\begin{corollary}\label{cor2}
The moduli space ${\mathcal H}_g(r)$ in \eqref{e26} is identified with the space of quadruples of the form
$(X,\, K^{1/2}_X,\, F,\, s)$, where $(X,\, K^{1/2}_X,\, F)\, \in\, {\mathcal B}_g(r)$ (see \eqref{e25}) and
$$
s\,\, \in\,\, \frac{H^0\left(3\mathbf{\Delta},\, \left((F\otimes K^{3/2}_X)\boxtimes (F^*\otimes K^{3/2}_X)
\otimes{\mathcal O}_{X\times X}(3{\mathbf \Delta})\right)\big\vert_{3\mathbf{\Delta}}\right)}{H^0
(X,\, {\rm ad}(F)\otimes K^{\otimes 2}_X)}
$$
such that $s_0\, :=\, s\big\vert_{\mathbf{\Delta}}\, =\,{\rm Id}_F$ (see \eqref{s0}).
\end{corollary}

\section{A canonical holomorphic section}

Recall the projection $\gamma$ in \eqref{e26} and the effective divisor $\Theta$ in \eqref{e27}.
In this section we we will construct a holomorphic map on the complement of the theta divisor
$$
{\mathbb S}\, :\, {\mathcal B}_g(r)\setminus \Theta\, \longrightarrow\, {\mathcal H}_g(r)
$$
such that $\gamma\circ{\mathbb S}\,=\, {\rm Id}_{{\mathcal B}_g(r)\setminus \Theta}$.

Take any $(X,\, K^{1/2}_X,\, F)\,\in\, {\mathcal B}_g(r)\setminus \Theta$. Since $$H^0(X,\, F\otimes K^{1/2}_X)\,
=\, 0\,=\, H^1(X,\, F\otimes K^{1/2}_X),$$
it follows that
\begin{equation}\label{e28a}
H^0(X\times X,\, (F\otimes K^{1/2}_X)\boxtimes (F^*\otimes K^{1/2}_X)\otimes{\mathcal O}_{X\times X}
({\mathbf\Delta}))\,=\, H^0(X,\, \text{End}(F\otimes K^{1/2}_X))\,=\, \text{End}(F)
\end{equation}
(see \cite[Remark 2.4]{BH1}, \cite[p.~8]{BH2}). As in \cite[(3.6)]{BH2}, let
\begin{equation}\label{e29}
\beta_F\, \in\, H^0(X\times X,\, (F\otimes K^{1/2}_X)\boxtimes (F^*\otimes K^{1/2}_X)\otimes{\mathcal O}_{X\times X}
({\mathbf\Delta}))
\end{equation}
be the section that corresponds to $\text{Id}_F\,\in\, H^0(X,\, \text{End}(F))$ by the isomorphism in
\eqref{e28a}.
The line bundle over $2{\mathbf\Delta}$
$$
K^{1/2}_X\boxtimes K^{1/2}_X\otimes{\mathcal O}_{X\times X}({\mathbf\Delta})\big\vert_{2{\mathbf\Delta}}
\,\longrightarrow\, 2{\mathbf\Delta}
$$
has a natural trivialization given by a section 
\begin{equation}\label{e29a}
\sigma_0\, \in\, H^0(2{\mathbf\Delta},\, K^{1/2}_X\boxtimes K^{1/2}_X\otimes{\mathcal O}_{X\times X}
({\mathbf\Delta})\big\vert_{2{\mathbf\Delta}})
\end{equation}
\cite[(3.8)]{BH2}, \cite[p.~688, Theorem 2.2]{BR}. Let
\begin{equation}\label{e30}
\widehat{\beta}_F\, \in\, H^0(2{\mathbf\Delta},\, (F\boxtimes F^*)\big\vert_{2{\mathbf\Delta}})
\end{equation}
be the section defined by the equation $\beta_F\big\vert_{2{\mathbf\Delta}}\,=\, \widehat{\beta}_F\otimes\sigma_0$,
where $\beta_F$ is the section in \eqref{e29} \cite[(3.9)]{BH2}. The
restriction of $\widehat{\beta}_F$ to ${\mathbf\Delta}\, \subset\, 2{\mathbf\Delta}$ is $\text{Id}_F$ \cite{BH2},
and hence $\widehat{\beta}_F$ defines a holomorphic connection on $F$ \cite[(3.10)]{BH2}; the
holomorphic connection on $F$ given by $\widehat{\beta}_F$ will be denoted by $\widehat{\beta}'_F$.

Any holomorphic connection on $X$ is integrable, because $\Omega^{2,0}_X\,=\, 0$. Therefore,
using the above holomorphic connection $\widehat{\beta}'_F$, for any
simply connected open subset $U\, \subset\, X$, the restriction $F\big\vert_U$ is canonically identified with
the trivialized holomorphic vector bundle $F_{x_0}\times U\, \longrightarrow\, U$, for any point $x_0\,\in\, U$.
More precisely, this identification of vector bundles is constructed by taking parallel translations of $F_{x_0}$,
for the integrable connection $\widehat{\beta}'_F$, along
paths originating from $x_0$. Consequently, the two holomorphic
vector bundles $q^*_1 F$ and $q^*_2 F$ are holomorphically identified over
an analytic neighborhood of ${\mathbf\Delta}\, \subset\, X\times X$. The restriction of this isomorphism to $2\mathbf{\Delta}$
coincides with the isomorphism
\begin{equation}\label{t1}
(q^*_1 F)\big\vert_{2{\mathbf\Delta}}\, \stackrel{\sim}{\longrightarrow}\, (q^*_2 F)\big\vert_{2{\mathbf\Delta}}
\end{equation}
given by $\widehat{\beta}_F$ in \eqref{e30}. Therefore, we have an extension of the isomorphism in \eqref{t1}
to an isomorphism
$$
(q^*_1 F)\big\vert_{m{\mathbf\Delta}}\, \stackrel{\sim}{\longrightarrow}\, (q^*_2 F)\big\vert_{m{\mathbf\Delta}}
$$
for every $m\, \geq\, 2$, which is given by $\widehat{\beta}'_F$. In particular, we get an isomorphism
\begin{equation}\label{e31}
f_3\,\, :\,\, (q^*_1 F)\big\vert_{3{\mathbf\Delta}}\, \stackrel{\sim}{\longrightarrow}\, (q^*_2 F)\big\vert_{3{\mathbf\Delta}}
\end{equation}
extending the isomorphism in \eqref{t1}.

Using the isomorphism $f_3$ in \eqref{e31}, the section $\beta_F\big\vert_{3{\mathbf\Delta}}$ in \eqref{e29} becomes
a section
\begin{gather}
\beta_{3,F}\,:=\,\beta_F\big\vert_{3{\mathbf\Delta}}
\, \in\, H^0\left((3{\mathbf\Delta},\,\left(\left((F\otimes F^*\otimes K^{1/2}_X)\boxtimes
K^{1/2}_X\right)\otimes{\mathcal O}_{X\times X}({\mathbf\Delta})\right)\big\vert_{3{\mathbf\Delta}}\right)\nonumber\\
=\, H^0\left(3{\mathbf\Delta},\,\left((K^{1/2}_X\boxtimes K^{1/2}_X)\otimes q^*_1\text{End}(F)
\otimes{\mathcal O}_{X\times X}({\mathbf\Delta})\right)\big\vert_{3{\mathbf\Delta}}\right)\, .\nonumber
\end{gather}
Now using the trace homomorphism
$$
q^*_1\text{End}(F)\, \longrightarrow\, q^*_1{\mathcal O}_X\,=\, {\mathcal O}_{X\times X}
\, , \ \ w\, \longmapsto\, \frac{1}{r}{\rm trace}(w)
$$
the above section $\beta_{3,F}$ produces a section
\begin{equation}\label{e32}
\sigma_1\, \in\, H^0(3{\mathbf\Delta},\, (K^{1/2}_X\boxtimes K^{1/2}_X\otimes{\mathcal O}_{X\times X}({\mathbf\Delta}))
\big\vert_{3{\mathbf\Delta}})\, ;
\end{equation}
the restriction of $\sigma_1$ to $2{\mathbf\Delta}\, \subset\, 3{\mathbf\Delta}$ coincides with the
section $\sigma_0$ in \eqref{e29a}.

Now define
\begin{gather}
\gamma_F\, :=\, (\beta_F)\big\vert_{3{\mathbf\Delta}}\otimes (\sigma_1)^{\otimes 2}
\, \in\, 
H^0\left(3{\mathbf\Delta},\, \left(\left((F\otimes K^{3/2}_X)\boxtimes (F^*\otimes K^{3/2}_X)\right)\otimes
{\mathcal O}_{X\times X}(3{\mathbf\Delta})\right)\big\vert_{3{\mathbf\Delta}}\right)\nonumber\\
\,=\, H^0\left(X,\, q_{1*}\left(\left(\left((F\otimes K^{3/2}_X)\boxtimes (F^*\otimes K^{3/2}_X)\right)\otimes
{\mathcal O}_{X\times X}(3{\mathbf\Delta})\right)\big\vert_{3{\mathbf\Delta}}\right)\right),\label{e33}
\end{gather}
where $\beta_F$ and $\sigma_1$ are the sections constructed in \eqref{e29} and \eqref{e32} respectively;
here the restriction of the projection $q_1$ to $3{\mathbf\Delta}\, \subset\, X\times X$ is also denoted by
$q_1$.

We have
$$
\left(\left((F\otimes K^{3/2}_X)\boxtimes (F^*\otimes K^{3/2}_X)\right)\otimes
{\mathcal O}_{X\times X}(3{\mathbf\Delta})\right)\big\vert_{3{\mathbf\Delta}}\,=\,
{\mathcal Q}_2(F\otimes K^{3/2}_X,\, F\otimes K^{-1/2}_X)
$$
(see \eqref{k1}), which implies that
$$
q_{1*}\left(\left(((F\otimes K^{3/2}_X)\boxtimes (F^*\otimes K^{3/2}_X))\otimes
{\mathcal O}_{X\times X}(3{\mathbf\Delta})\right)\big\vert_{3{\mathbf\Delta}}\right)\,=\,
{\mathcal K}_2(F\otimes K^{3/2}_X,\, F\otimes K^{-1/2}_X)
$$
(see \eqref{k2}). Therefore, from \eqref{k3} it follows that 
\begin{gather}
H^0\left(X,\, q_{1*}\left(\left(((F\otimes K^{3/2}_X)\boxtimes (F^*\otimes K^{3/2}_X))\otimes
{\mathcal O}_{X\times X}(3{\mathbf\Delta})\right)\big\vert_{3{\mathbf\Delta}}\right)\right)\nonumber\\
=\, H^0(X,\, {\mathcal K}_2(F\otimes K^{3/2}_X,\, F\otimes K^{-1/2}_X))\,=\,
H^0(X,\, \text{Diff}^2_X(F\otimes K^{-1/2}_X,\, F\otimes K^{3/2}_X))\, .\nonumber
\end{gather}
Using this isomorphism, the section $\gamma_F$ constructed in \eqref{e33} gives a holomorphic differential
operator
\begin{equation}\label{e34}
\widehat{\gamma}_F\,\in\, H^0(X,\, \text{Diff}^2_X(F\otimes K^{-1/2}_X,\, F\otimes K^{3/2}_X))\, .
\end{equation}
{}From the symbol homomorphism constructed in \eqref{k5} it can be shown that the symbol of
the differential operator $\widehat{\gamma}_F$ in \eqref{e34} is
$$
{\rm Id}_F\, \in\, H^0(X,\, {\rm Hom}(F\otimes K^{-1/2}_X,\, F\otimes K^{3/2}_X)\otimes
(TX)^{\otimes 2})\,=\, H^0(X,\, \text{End}(F))\, .
$$
Indeed, in view of the construction of $\gamma_F$ in \eqref{e33}, this is a consequence the following two facts:
\begin{itemize}
\item The restriction of the section $\widehat{\beta}_F$ in \eqref{e30} to
${\mathbf\Delta}\,\subset\, 2{\mathbf\Delta}$ is $\text{Id}_F$, and

\item the restriction of the section $\sigma_1$ in \eqref{e32} to
${\mathbf\Delta}\,\subset\, 3{\mathbf\Delta}$ is the constant function $1$ on
$\bf\Delta$ (recall that the restriction of $\sigma_1$ to $2{\mathbf\Delta}\,\subset\, 3{\mathbf\Delta}$
coincides with $\sigma_0$ in \eqref{e29a}).
\end{itemize}

Recall the space of differential operators $\widetilde{\mathcal D}'(X,\, K^{1/2}_X,\, F)$
defined in the proof of
Lemma \ref{lem6}. Since the symbol of the differential operator $\widehat{\gamma}_F$
in \eqref{e34} is $\text{Id}_F$, it is an element of $\widetilde{\mathcal D}'
(X,\, K^{1/2}_X,\, F)$. Therefore, $\widehat{\gamma}_F$ gives an element of the
quotient space $\widetilde{\mathcal D}(X,\, K^{1/2}_X,\, F)$ of
$\widetilde{\mathcal D}'(X,\, K^{1/2}_X,\, F)$ (see Definition \ref{ds}). Let
\begin{equation}\label{e35}
\widehat{\gamma}'_F\,\in\, \widetilde{\mathcal D}(X,\, K^{1/2}_X,\, F)
\end{equation}
be the element given by $\widehat{\gamma}_F$.

The above construction is summarized in the following lemma.

\begin{lemma}\label{lem8}
There is a natural holomorphic map
$$
{\mathbb S}\, :\, {\mathcal B}_g(r)\setminus \Theta\, \longrightarrow\, {\mathcal H}_g(r)
$$
such that $\gamma\circ{\mathbb S}\,=\, {\rm Id}_{{\mathcal B}_g(r)\setminus \Theta}$, where
$\gamma$ is the projection in \eqref{e26}.
\end{lemma}

\begin{proof}
Take any $(X,\, K^{1/2}_X,\, F)\,\in\, {\mathcal B}_g(r)\setminus \Theta$. From 
Theorem \ref{thm2} we know that $$(X,\, K^{1/2}_X,\, F,\, \widehat{\gamma}'_F)\, \in\, {\mathcal H}_g(r),$$
where $\widehat{\gamma}'_F$ is constructed from $(X,\, K^{1/2}_X,\, F)$ in \eqref{e35}.
Let
$$
{\mathbb S}\, :\, {\mathcal B}_g(r)\setminus \Theta\, \longrightarrow\, {\mathcal H}_g(r)
$$
be the map that sends any $(X,\, K^{1/2}_X,\, F)\,\in\, {\mathcal B}_g(r)\setminus \Theta$
to $(X,\, K^{1/2}_X,\, F,\, \widehat{\gamma}'_F)$. 
It is evident that $\gamma\circ{\mathbb S}\,=\, {\rm Id}_{{\mathcal B}_g(r)\setminus \Theta}$.
\end{proof}

The above map ${\mathbb S}$ can also be described using Corollary \ref{cor2}.

\section{A canonical isomorphism of torsors and a symplectic structure}

\subsection{A canonical isomorphism of torsors}

Let
\begin{equation}\label{e36}
{\mathcal L}\,:=\, {\mathcal O}_{{\mathcal B}_g(r)}(\Theta)\, \longrightarrow\, {\mathcal B}_g(r)
\end{equation}
be the line bundle corresponding to the reduced divisor $\Theta$ in \eqref{e27}. Consider the short
exact sequence
\begin{equation}\label{atl}
0\, \longrightarrow\,{\mathcal O}_{{\mathcal B}_g(r)}\,=\, {\rm Diff}^0_X({\mathcal L},\, {\mathcal L})
\, \longrightarrow\,{\rm At}({\mathcal L})\, :=\, {\rm Diff}^1_X({\mathcal L},\, {\mathcal L})
\, \stackrel{\zeta}{\longrightarrow}\,T{\mathcal B}_g(r) \, \longrightarrow\,0\,,
\end{equation}
where $\zeta$ is the symbol map. We note that \eqref{atl} coincides with
the Atiyah exact sequence for the line bundle $\mathcal L$ in \eqref{e36} (see \eqref{e11}). Let
\begin{equation}\label{e36a}
0\, \longrightarrow\,T^*{{\mathcal B}_g(r)}\, \longrightarrow\,{\rm At}({\mathcal L})^*
\, \stackrel{\eta}{\longrightarrow}\,{\mathcal O}_{{\mathcal B}_g(r)} \, \longrightarrow\,0
\end{equation}
be the dual of the sequence in \eqref{atl}. Define
\begin{equation}\label{e37}
{\mathcal C}({\mathcal L})\,:=\, \eta^{-1}({\rm image}({\mathbf 1}))\, \subset\, {\rm At}({\mathcal L})^*\, ,
\end{equation}
where $\eta$ is the projection in \eqref{e36a}, and $\mathbf 1$ is the section of
${\mathcal O}_{{\mathcal B}_g(r)}$ given by the constant function $1$ on ${\mathcal B}_g(r)$. Let
\begin{equation}\label{e38}
\rho\,\, :\,\, {\mathcal C}({\mathcal L})\, \longrightarrow\, {\mathcal B}_g(r)
\end{equation}
be the restriction of the natural projection ${\rm At}({\mathcal L})^*\,\longrightarrow\, {\mathcal B}_g(r)$. From
\eqref{e36a} it follows that ${\mathcal C}({\mathcal L})$ is an algebraic torsor
over ${\mathcal B}_g(r)$ for the holomorphic cotangent bundle $T^*{{\mathcal B}_g(r)}$. Giving a
holomorphic section of ${\mathcal C}({\mathcal L})$ over an open subset $U\, \subset\, {\mathcal B}_g(r)$
is equivalent to giving a holomorphic connection on the line bundle ${\mathcal L}\big\vert_U
\, \longrightarrow\, U$.

Recall from Theorem \ref{thm1} the algebraic torsor $\gamma\, :\, {\mathcal H}_g(r)\, \longrightarrow
\, {\mathcal B}_g(r)$ for the holomorphic cotangent bundle $T^*{{\mathcal B}_g(r)}$. Using
Lemma \ref{lem8} it can be deduced that the two torsors ${\mathcal C}({\mathcal L})$ (constructed in \eqref{e38})
and ${\mathcal H}_g(r)$, over ${\mathcal B}_g(r)$ for the holomorphic cotangent bundle $T^*{{\mathcal B}_g(r)}$,
are naturally identified when restricted to the open subset ${\mathcal B}_g(r)
\setminus \Theta \, \subset\, {\mathcal B}_g(r)$, where $\Theta$ is the divisor in \eqref{e27}. To explain this,
define
\begin{gather}
{\mathcal H}^0_g(r)\,:=\, \gamma^{-1}({\mathcal B}_g(r) \setminus \Theta)\,\subset\, {\mathcal H}_g(r)\label{e39}\\
{\mathcal C}({\mathcal L})^0\,:=\, \rho^{-1}({\mathcal B}_g(r) \setminus \Theta)\,\subset\,
{\mathcal C}({\mathcal L})\label{e40}
\end{gather}
where $\gamma$ and $\rho$ are the projections in \eqref{e26} and \eqref{e38} respectively. Since the
restriction of ${\mathcal L}$ to ${\mathcal B}_g(r) \setminus \Theta
\, \subset\, {\mathcal B}_g(r)$ is the trivial bundle ${\mathcal O}_{{\mathcal B}_g(r)\setminus\Theta}$,
there is a natural integrable algebraic connection $\nabla^0$ on the restriction of 
${\mathcal L}$ to ${\mathcal B}_g(r) \setminus \Theta$ which is given by the de Rham differential
$d$ on ${\mathcal O}_{{\mathcal B}_g(r)\setminus\Theta}$.
This connection $\nabla^0$ produces an algebraic
splitting of the Atiyah exact sequence for ${\mathcal L}\big\vert_{{\mathcal B}_g(r) \setminus \Theta}$
(see \eqref{atl}), which in turn
produces an algebraic splitting, over ${\mathcal B}_g(r) \setminus \Theta$, of the short
exact sequence in \eqref{e36a}. Hence we get an algebraic section
\begin{equation}\label{e41}
G\,\, :\,\, {\mathcal B}_g(r) \setminus \Theta\, \longrightarrow\, {\mathcal C}({\mathcal L})^0
\end{equation}
of the bundle $\rho\, :\, {\mathcal C}({\mathcal L})^0\, \longrightarrow\, {\mathcal B}_g(r) \setminus \Theta$
in \eqref{e40}. To explain the construction of $G$, if
$$
s_0\, :\, {\mathcal O}_{{\mathcal B}_g(r) \setminus \Theta}\, \longrightarrow\,
{\rm At}({\mathcal L})^*\big\vert_{{\mathcal B}_g(r) \setminus \Theta}
$$
is the splitting homomorphism for \eqref{e36a} over ${\mathcal B}_g(r) \setminus \Theta$, then $G(z)\,=\, s_0({\mathbf 1}(z))$
for all $z\, \in\, {\mathcal B}_g(r) \setminus \Theta$, where ${\mathbf 1}$ as before is
the section of ${\mathcal O}_{{\mathcal B}_g(r) \setminus \Theta}$ given by the constant function $1$ on
${\mathcal B}_g(r) \setminus \Theta$.

Consider ${\mathcal H}^0_g(r)$ and ${\mathcal C}({\mathcal L})^0$
constructed in \eqref{e39} and \eqref{e40} respectively. We have a map
\begin{equation}\label{e42}
H\,\, :\,\, {\mathcal H}^0_g(r)\, \longrightarrow\, {\mathcal C}({\mathcal L})^0
\end{equation}
that sends any $v\, \in\, {\mathcal H}^0_g(r)$ to
\begin{equation}\label{e43}
G(\gamma(v))+(v-{\mathbb S}(\gamma(v)))\, ,
\end{equation}
where $\mathbb S$, $G$ and $\gamma$ are constructed in Lemma \ref{lem8}, \eqref{e41} and \eqref{e26} respectively;
note that $v-{\mathbb S}(\gamma(v))$ in \eqref{e43} is an element of $T^*_{\gamma(v)}{\mathcal B}_g(r)$ (recall
that ${\mathcal H}_g(r)$ is a torsor for $T^*{\mathcal B}_g(r)$ by Theorem \ref{thm1}), so the sum in
\eqref{e43} is an element of ${\mathcal C}({\mathcal L})^0$ because ${\mathcal C}({\mathcal L})^0$ is a torsor
over ${\mathcal B}_g(r)\setminus\Theta$ for $T^*({\mathcal B}_g(r)\setminus\Theta)$ (see \eqref{e37} and \eqref{e40}).

It is evident that $H$ in \eqref{e42} is an algebraic isomorphism of torsors over
${\mathcal B}_g(r)\setminus\Theta$ for $T^*({\mathcal B}_g(r)\setminus\Theta)$.

\begin{theorem}\label{thm3}
The algebraic isomorphism $H$ in \eqref{e42}, of torsors
over ${\mathcal B}_g(r)\setminus\Theta$ for $T^*({\mathcal B}_g(r)\setminus\Theta)$, extends to an
algebraic isomorphism of $T^*{\mathcal B}_g(r)$--torsors
$$
{\mathcal H}\,\, :\,\, {\mathcal H}_g(r)\, \longrightarrow\, {\mathcal C}({\mathcal L})
$$
over entire ${\mathcal B}_g(r)$.
\end{theorem}

\begin{proof}
Consider the submersion $\beta$ in \eqref{e25}. The kernel of the differential of $\beta$
$$
d\beta\, :\, T{\mathcal B}_g(r)\, \longrightarrow\, \beta^*T{\mathcal M}^\theta_g
$$
will be denoted by $T_\beta$; in other words, $T_\beta$ is the relative tangent bundle
for the projection $\beta$. So we have the short exact sequence of vector bundles
$$
0 \, \longrightarrow\, \beta^*T^*{\mathcal M}^\theta_g \,\xrightarrow{(d\beta)^*}\, 
T^*{\mathcal B}_g(r) \, \longrightarrow\, T^*_\beta\, :=\, (T_\beta)^* \, \longrightarrow\,0
$$
over ${\mathcal B}_g(r)$. Let ${\mathcal W}\, \longrightarrow\, {\mathcal B}_g(r)$ be an algebraic
torsor over ${\mathcal B}_g(r)$ for $T^*{\mathcal B}_g(r)$. Then
\begin{equation}\label{t2}
{\mathcal W}/(\beta^*T^*{\mathcal M}^\theta_g)\, \longrightarrow\, {\mathcal B}_g(r)
\end{equation}
is an algebraic torsor over ${\mathcal B}_g(r)$ for the vector bundle $T^*_\beta$.

Substitute the two $T^*{\mathcal B}_g(r)$--torsors ${\mathcal H}_g(r)$ and ${\mathcal C}({\mathcal L})$
in place of the above $T^*{\mathcal B}_g(r)$--torsor $\mathcal W$. The construction in \eqref{t2} produces
two $T^*_\beta$--torsors over ${\mathcal B}_g(r)$ from ${\mathcal H}_g(r)$ and ${\mathcal C}({\mathcal L})$;
these two $T^*_\beta$--torsors will
be denoted by $\widehat{\mathcal H}_g(r)$ and $\widehat{\mathcal C}({\mathcal L})$ respectively.
The restrictions of $\widehat{\mathcal H}_g(r)$ and $\widehat{\mathcal C}({\mathcal L})$ to the
Zariski open subset $${\mathcal B}_g(r)\setminus \Theta\, \subset\, {\mathcal B}_g(r)$$ will be denoted by
$\widehat{\mathcal H}^0_g(r)$ and $\widehat{\mathcal C}({\mathcal L})^0$ respectively.

The isomorphism $H$ in \eqref{e42} produces an algebraic isomorphism of $T^*_\beta$--torsors over
${\mathcal B}_g(r)\setminus \Theta$
\begin{equation}\label{t3}
\widehat{H}\,\, :\,\, \widehat{\mathcal H}^0_g(r)\, \longrightarrow\, \widehat{\mathcal C}({\mathcal L})^0\, .
\end{equation}
For each point $z\,\in\, {\mathcal M}^\theta_g$, consider the restriction of the isomorphism $\widehat H$ in \eqref{t3}
to the complement $\beta^{-1}(z)\setminus (\beta^{-1}(z)\cap\Theta)$, where $\beta$ is the projection in
\eqref{e25}. This restriction coincides with the isomorphism constructed in \cite[Lemma 3.1]{BH2}. Therefore, from
\cite[Corollary 4.5]{BH2} we know that $\widehat{H}$ in \eqref{t3} extends to an algebraic isomorphism
of $T^*_\beta$--torsors
\begin{equation}\label{t4}
\widehat{H}'\,\, :\,\, \widehat{\mathcal H}_g(r)\, \longrightarrow\, \widehat{\mathcal C}({\mathcal L})\, .
\end{equation}
over entire ${\mathcal B}_g(r)$. Let
\begin{equation}\label{o1}
{\mathbb G}\,\, :=\,\, \{(y,\, \widehat{H}'(y))\, \mid\, y\, \in\, \widehat{\mathcal H}_g(r)\}\, \subset\,
\widehat{\mathcal H}_g(r)\times_{{\mathcal B}_g(r)} \widehat{\mathcal C}({\mathcal L})
\end{equation}
be the graph of the map $\widehat{H}'$ in \eqref{t4}.

Consider the natural quotient map
$$
{\mathcal H}_g(r)\times_{{\mathcal B}_g(r)}{\mathcal C}({\mathcal L})\, \stackrel{{\mathbf q}_1}{\longrightarrow}\,
({\mathcal H}_g(r)/(\beta^*T^*{\mathcal M}^\theta_g))
\times_{{\mathcal B}_g(r)} ({\mathcal C}({\mathcal L})/(\beta^*T^*{\mathcal M}^\theta_g))\,=\,
\widehat{\mathcal H}_g(r)\times_{{\mathcal B}_g(r)} \widehat{\mathcal C}({\mathcal L})\, .
$$
The inverse image
$$
{\mathcal G}_1\, :=\, {\mathbf q}^{-1}_1({\mathbb G})\, \subset\,
{\mathcal H}_g(r)\times_{{\mathcal B}_g(r)}{\mathcal C}({\mathcal L})\, ,
$$
where ${\mathbb G}$ is constructed in \eqref{o1}, is an algebraic torsor over ${\mathcal B}_g(r)$ for
the vector bundle $(\beta^*T^*{\mathcal M}^\theta_g)^{\oplus 2}$. Let
\begin{equation}\label{o2}
{\mathcal G}\,:=\, {\mathcal G}_1/(\beta^*T^*{\mathcal M}^\theta_g)\, \stackrel{\mathbf q}{\longrightarrow}
\, {\mathcal B}_g(r)
\end{equation}
be the quotient for the diagonal action of $\beta^*T^*{\mathcal M}^\theta_g$; so for any
$z\, \in\, {\mathcal B}_g(r)$, two elements $(x,\, y)$ and $(x',\, y')$ of the fiber
$({\mathcal G}_1)_z$ give the same element of the fiber ${\mathcal G}_z$ if and only if there is an element
$v\, \in\, (\beta^*T^*{\mathcal M}^\theta_g)_z\,=\, T^*_{\beta(z)}{\mathcal M}^\theta_g$ such that
$x'\,=\, x+v$ and $y'\,=\, y+v$. Therefore, we have a map
$$
\nu\,\, :\,\,{\mathcal H}_g(r)\times_{{\mathcal B}_g(r)} {\mathcal G}\,\, \longrightarrow\,
\,{\mathcal C}({\mathcal L})
$$
that sends any $(h,\, (x,\, y))$, where
$h\, \in\, {\mathcal H}_g(r)_z$,\, $(x,\, y)\,\in\, {\mathcal G}_z$ and $z\,\in\,
{\mathcal B}_g(r)$, to $$y+(h-x)\,\, \in\,\, {\mathcal C}({\mathcal L})_z$$ (note that
$h-x\,\in\, T^*_z{\mathcal B}_g(r)$). Using this map $\nu$,
the space of holomorphic sections of ${\mathcal G}$ over any open subset $U\, \subset\, {\mathcal B}_g(r)$
is identified with the space of holomorphic isomorphisms ${\mathcal H}_g(r)\big\vert_U \, \longrightarrow\,
{\mathcal C}({\mathcal L})\big\vert_U$, of the torsor over $U$ for $T^*U$, that induce the isomorphism
$\widehat{H}'\big\vert_U$ in \eqref{t4} of $T^*_\beta\big\vert_U$--torsors.

We note that ${\mathcal G}$ is an algebraic torsor over ${\mathcal B}_g(r)$ for the vector bundle
$\beta^*T^*{\mathcal M}^\theta_g$, where $\beta$ is the projection in
\eqref{e25}, as follows: For any $z\, \in\, {\mathcal B}_g(r)$, take
$(x,\, y)\,\in\, {\mathcal G}_z$ and $v\, \in\, (\beta^*T^*{\mathcal M}^\theta_g)_z\,=\,
T^*_{\beta(z)}{\mathcal M}^\theta_g$; then we have
$$
(x,\, y) +v\,=\, (x+v,\, y-v)\, .
$$

The moduli space ${\mathcal B}_g(r)$ has an algebraic involution
\begin{equation}\label{o3}
{\mathcal I}_B\, \, :\, \, {\mathcal B}_g(r)\, \longrightarrow\, {\mathcal B}_g(r)\, ,\ \ \,
(X,\, K^{1/2}_X,\, F)\, \longmapsto\, (X,\, K^{1/2}_X,\, F^*)\, .
\end{equation}
This involution ${\mathcal I}_B$ preserves the divisor $\Theta$ in \eqref{e27} (see \eqref{o-1}). Hence the
action ${\mathbb Z}/2{\mathbb Z}$ on ${\mathcal B}_g(r)$ given by ${\mathcal I}_B$ lifts to the line bundle
$\mathcal L$ in \eqref{e36}. Consequently, the involution ${\mathcal I}_B$ lifts to an involution
\begin{equation}\label{o4}
{\mathcal I}_C\, \, :\, \, {\mathcal C}({\mathcal L})\, \longrightarrow\, {\mathcal C}({\mathcal L})
\end{equation}
of ${\mathcal C}({\mathcal L})$.

We will now describe the relationship between ${\mathcal I}_C$ and
the $T^*{\mathcal B}_g(r)$--torsor structure of ${\mathcal C}({\mathcal L})$. Take any $y\, \in\, {\mathcal B}_g(r)$,
$z\, \in\, {\mathcal C}({\mathcal L})_y$ and $w\, \in\, T^*_y{\mathcal B}_g(r)$. Then we have
\begin{equation}\label{o5}
{\mathcal I}_C(z+w)\,=\, {\mathcal I}_C(z) - (d{\mathcal I}_B)^*_{{\mathcal I}_B(y)} (w)\, ,
\end{equation}
where $(d{\mathcal I}_B)^*_{{\mathcal I}_B(y)}\, :\, T^*_y{\mathcal B}_g(r)\, \longrightarrow\,
T^*_{{\mathcal I}_B(y)}{\mathcal B}_g(r)$ is the dual of the differential $(d{\mathcal I}_B)_{{\mathcal I}_B(y)}\,:\,
T_{{\mathcal I}_B(y)}{\mathcal B}_g(r)\, \longrightarrow\, T_y{\mathcal B}_g(r)$
of the map ${\mathcal I}_B$ at the point ${\mathcal I}_B(y)$.

Since $\bigwedge^2J^1(K^{-1/2}_X)\,=\, K_X\otimes (K^{-1/2}_X)^{\otimes 2}\,=\, {\mathcal O}_X$ (see
\eqref{ej}), it follows that $J^1(K^{-1/2}_X)\,=\, J^1(K^{-1/2}_X)^*$. Let $F$ be a stable vector bundle
on $X$ of rank $r$ and degree zero. Since $F$ admits a holomorphic connection (see the proof of
Proposition \ref{prop1}), from Proposition \ref{prop-1} we conclude that
$$
J^1(F^*\otimes K^{-1/2}_X)\,=\, F^*\otimes J^1(K^{-1/2}_X)\,=\, F^*\otimes J^1(K^{-1/2}_X)^*
\,=\, J^1(F\otimes K^{-1/2}_X)^*\, .
$$
Fixing an isomorphism of $J^1(F^*\otimes K^{-1/2}_X)$ with $J^1(F\otimes K^{-1/2}_X)^*$ we conclude that any
holomorphic connection on $J^1(F\otimes K^{-1/2}_X)$ produces a holomorphic connection on
$J^1(F\otimes K^{-1/2}_X)^*\,=\, J^1(F^*\otimes K^{-1/2}_X)$. It is straightforward to check that this
produces a bijection
$$
M_F\,\, :\,\, {\mathcal D}(F)\, \longrightarrow\, {\mathcal D}(F^*)\, ,
$$
where ${\mathcal D}(F)$ is constructed in \eqref{qe}. This map $M_F$ does not depend on the choice of the
isomorphism of $J^1(F^*\otimes K^{-1/2}_X)$ with $J^1(F\otimes K^{-1/2}_X)^*$; this is because
in the construction of ${\mathcal D}(F)$ quotienting by $\text{Aut}(F\otimes K^{-1/2}_X)$ was executed.
Now we have an algebraic involution
\begin{equation}\label{o6}
{\mathcal I}_H\, \, :\, \, {\mathcal H}_g(r)\, \longrightarrow\,{\mathcal H}_g(r)\, ,\ \ \,
(X,\, K^{1/2}_X,\, F,\, D)\, \longmapsto\, (X,\, K^{1/2}_X,\, F^*,\, M_F(D))\, ,
\end{equation}
where $M_F$ is the map constructed above. It is evident that
$$
\gamma\circ {\mathcal I}_H\,=\, {\mathcal I}_B\circ \gamma\, ,
$$
where $\gamma$ and ${\mathcal I}_B$ are the maps in \eqref{e26} and \eqref{o3} respectively.

We will describe the relationship between ${\mathcal I}_H$ and
the $T^*{\mathcal B}_g(r)$--torsor structure of ${\mathcal H}_g(r)$. Take any $y\, \in\, {\mathcal B}_g(r)$,
$z\, \in\, {\mathcal H}_g(r)_y$ and $w\, \in\, T^*_y{\mathcal B}_g(r)$. Then we have
\begin{equation}\label{o7}
{\mathcal I}_H(z+w)\,=\, {\mathcal I}_H(z) - (d{\mathcal I}_B)^*_{{\mathcal I}_B(y)} (w)\, ,
\end{equation}
where $(d{\mathcal I}_B)^*_{{\mathcal I}_B(y)}$ is the homomorphism in \eqref{o5}.

We note that the projection $\beta$ in \eqref{e25} satisfies the equation
$$\beta\circ {\mathcal I}_B\,=\, \beta\, .$$
In view of this, from \eqref{o5} and \eqref{o7} we conclude the following:
\begin{enumerate}
\item The involution ${\mathcal I}_B$ lifts to an involution
\begin{equation}\label{o8}
{\mathcal I}_G\, \,:\,\, {\mathcal G}\, \longrightarrow\, {\mathcal G}
\end{equation}
of ${\mathcal G}$ constructed in \eqref{o2}. For any $z\, \in\, {\mathcal B}_g(r)$, take
$(x,\, y)\,\in\, {\mathcal G}_z$; the map ${\mathcal I}_G$ sends $(x,\, y)$ to
$({\mathcal I}_H(x),\, {\mathcal I}_C(y))$, where ${\mathcal I}_H$ and ${\mathcal I}_C$ are
the involutions in \eqref{o6} and \eqref{o4} respectively.

\item For any $v\, \in\, (\beta^*T^*{\mathcal M}^\theta_g)_z\,=\,
T^*_{\beta(z)}{\mathcal M}^\theta_g$,
\begin{equation}\label{o9}
{\mathcal I}_G((x,y) +v)\,\,=\,\,{\mathcal I}_G((x,y)) +v\, ;
\end{equation}
recall that ${\mathcal G}$ is a torsor for $\beta^*T^*{\mathcal M}^\theta_g$.
\end{enumerate}

Take any point
$$
t\, \,:=\,\, (X,\, K^{1/2}_X)\,\, \in\,\,{\mathcal M}^\theta_g\, .
$$
Let
\begin{equation}\label{o10}
{\mathcal G}^t\, \,:=\, {\mathcal G}\big\vert_{\beta^{-1}(t)} \, \stackrel{{\mathbf q}^t}{\longrightarrow}
\, {\mathcal B}^t\, :=\, \beta^{-1}(t)
\end{equation}
be the restriction, where ${\mathbf q}^t$ is the restriction of the projection $\mathbf q$ in \eqref{o2},
and $\beta$ is the projection in \eqref{e25}. So ${\mathcal G}^t$ is an algebraic torsor over ${\mathcal B}^t$ for the
trivial vector bundle $${\mathcal V}^t\, :=\, {\mathcal B}^t\times T^*_t{\mathcal M}^\theta_g\,
\longrightarrow\, {\mathcal B}^t$$
over ${\mathcal B}^t$ with fiber $T^*_t{\mathcal M}^\theta_g$; this is because
${\mathcal G}$ is a torsor for $\beta^*T^*{\mathcal M}^\theta_g$.

Let
$$
{\mathcal I}^t_B\, :=\, {\mathcal I}_B\big\vert_{{\mathcal B}^t}
\, \, :\, \, {\mathcal B}^t\, \longrightarrow\,{\mathcal B}^t
$$
be the restriction of ${\mathcal I}_B$ in \eqref{o3} to the subvariety ${\mathcal B}^t$ in \eqref{o10}.

We note that the isomorphism classes of algebraic ${\mathcal V}^t$--torsors over the variety
${\mathcal B}^t$ in \eqref{o10} are parametrized by
$$
H^1({\mathcal B}^t,\, {\mathcal V}^t)\,=\, H^1({\mathcal B}^t,\,
{\mathcal O}_{{\mathcal B}^t})\otimes T^*_t{\mathcal M}^\theta_g\, .
$$
Let
\begin{equation}\label{o11}
{\mathbf c}\,\,\in\,\, H^1({\mathcal B}^t,\,
{\mathcal O}_{{\mathcal B}^t})\otimes T^*_t{\mathcal M}^\theta_g
\end{equation}
be the class of the ${\mathcal V}^t$--torsor ${\mathcal G}^t$ in \eqref{o10}. From \eqref{o9}
it follows immediately that
\begin{equation}\label{o12}
({\mathcal I}^t_B)^*{\mathbf c}\,\,=\,\, {\mathbf c}\, .
\end{equation}

Consider the determinant map
$$
\widetilde{\delta}\, :\, {\mathcal B}^t\, \longrightarrow\, \text{Pic}^0(X)\, ,\ \ \, F\, \longmapsto\,
\bigwedge\nolimits^r F\, .
$$
The corresponding homomorphism
\begin{equation}\label{dc}
\widetilde{\delta}^*\,:\, H^1(\text{Pic}^0(X),\, {\mathcal O}_{\text{Pic}^0(X)})
\, \longrightarrow\, H^1({\mathcal B}^t,\,
{\mathcal O}_{{\mathcal B}^t})
\end{equation}
is an isomorphism. Indeed, $\text{Pic}({\mathcal B}^t) \,=\, \text{Pic}(\text{Pic}^0(X))\oplus {\mathbb Z}$
\cite[p.~57, Theorem D]{DN} when $r \, \geq\,2$ and $(g,\, r)\, \not=\, (2,\, 2)$ (when $g\,=\,2\,=\, r$, the
quotient $\text{Pic}({\mathcal B}^t)/\text{Pic}(\text{Pic}^0(X))$ is a finite group), so using the
exact sequence of cohomologies
$$
H^1(Y,\, 2\pi\sqrt{-1}\cdot{\mathbb Z})\, \longrightarrow\, H^1(Y,\, {\mathcal O}_Y) \, \longrightarrow\,
H^1(Y,\, {\mathcal O}^*_Y) \, \longrightarrow\, H^2(Y,\, 2\pi\sqrt{-1}\cdot{\mathbb Z})
$$
associated to the exponential sequence
$$
0\, \longrightarrow\, 2\pi\sqrt{-1}\cdot{\mathbb Z} \, \longrightarrow\,{\mathcal O}_Y \,
\xrightarrow{\,\exp\, \,}\,{\mathcal O}^*_Y \, \longrightarrow\,0
$$
on a complex variety $Y$ we conclude that $\widetilde{\delta}^*$ in \eqref{dc} is an isomorphism.
We also note that the involution
$$
\text{Pic}^0(X)\, \longrightarrow\, \text{Pic}^0(X)\, ,\ \ \, \xi\, \longmapsto\, \xi^*
$$
acts on $H^1(\text{Pic}^0(X),\, {\mathcal O}_{\text{Pic}^0(X)})$ as multiplication by $-1$. In other
words, no nonzero element of $H^1(\text{Pic}^0(X),\, {\mathcal O}_{\text{Pic}^0(X)})$ is fixed by this
involution. In view of these, from \eqref{o12} we conclude that
$$
{\mathbf c}\,=\, 0\, .
$$
So from \eqref{o11} we conclude that the ${\mathcal V}^t$--torsor ${\mathcal G}^t$ in \eqref{o10}
is the trivial ${\mathcal V}^t$--torsor
$$
{\mathcal V}^t\,=\, {\mathcal B}^t\times T^*_t{\mathcal M}^\theta_g \, \longrightarrow\, {\mathcal B}^t\, .
$$

Restrict $H$ (constructed in \eqref{e42}) to ${\mathcal B}^t\, \subset\,{\mathcal B}_g(r)$;
denote this restriction by $H^t$. From the above isomorphism of ${\mathcal G}^t$ with the
trivial ${\mathcal V}^t$--torsor ${\mathcal V}^t$ it follows that
$H^t$ is a meromorphic function on ${\mathcal B}^t$ with values in the vector space
$T^*_t{\mathcal M}^\theta_g$ (recall that ${\mathcal V}^t\, :=\,
{\mathcal B}^t\times T^*_t{\mathcal M}^\theta_g$). This meromorphic function is evidently regular
on the complement ${\mathcal B}^t\setminus ({\mathcal B}^t\cap\Theta)$, where $\Theta$ is
constructed in \eqref{e27}. From the construction of $H$ it is straightforward to deduce that
$H^t$ has a pole of order at most one on the divisor ${\mathcal B}^t\cap\Theta\, \subset\,
{\mathcal B}^t$. On the other hand, we know that
$$
H^0({\mathcal B}^t,\, {\mathcal O}_{{\mathcal B}^t}({\mathcal B}^t\cap\Theta))\,=\,
H^0({\mathcal B}^t,\, {\mathcal O}_{{\mathcal B}^t})\,=\, \mathbb C
$$
\cite[p.~169, Theorem 2]{BNR}. Consequently, the section $H^t$ over ${\mathcal B}^t\setminus ({\mathcal B}^t\cap\Theta)$
extends to entire ${\mathcal B}^t$ as a regular section. From this it follows immediately
that the isomorphism $H$ of torsors
over ${\mathcal B}_g(r)\setminus\Theta$ for $T^*({\mathcal B}_g(r)\setminus\Theta)$, extends to an
algebraic isomorphism of $T^*{\mathcal B}_g(r)$--torsors
${\mathcal H}_g(r)\, \longrightarrow\, {\mathcal C}({\mathcal L})$
over entire ${\mathcal B}_g(r)$.
\end{proof}

\subsection{A holomorphic symplectic form}

Recall the holomorphic line bundle ${\mathcal L}\, \longrightarrow\, {\mathcal B}_g(r)$ in \eqref{e36}.
The space ${\mathcal C}({\mathcal L})$ in \eqref{e38} has a canonical holomorphic symplectic structure.
We will briefly recall the construction of this symplectic form on ${\mathcal C}({\mathcal L})$.

From the construction of ${\mathcal C}({\mathcal L})$ in \eqref{e37} it follows immediately that there is
a tautological holomorphic splitting
$$
\rho^*\text{At}({\mathcal L})\,=\, \rho^*{\mathcal O}_{{\mathcal B}_g(r)}\oplus\rho^* T{\mathcal B}_g(r)
\,=\, {\mathcal O}_{{\mathcal C}({\mathcal L})} \oplus\rho^* T{\mathcal B}_g(r)\, ,
$$
where $\rho$ is constructed in \eqref{e38}. This decomposition of $\rho^*\text{At}({\mathcal L})$
gives a holomorphic projection
\begin{equation}\label{f0}
f_0\, :\, \rho^*\text{At}({\mathcal L})\, \longrightarrow\, {\mathcal O}_{{\mathcal C}({\mathcal L})}\, .
\end{equation}
Let
\begin{equation}\label{aes2}
0\, \longrightarrow\, {\mathcal O}_{{\mathcal C}({\mathcal L})}\, \longrightarrow\, 
\text{At}(\rho^*{\mathcal L})\, \stackrel{\zeta'}{\longrightarrow}\, T{\rm Conn}({\mathcal L})\, \longrightarrow\, 0
\end{equation}
be the Atiyah exact sequence for $\rho^*{\mathcal L}$ (see \eqref{e11}). We also have a tautological projection
$$
h_0\, :\, \text{At}(\rho^* {\mathcal L})\, \longrightarrow\, \rho^*\text{At}({\mathcal L})
$$
such that the diagram
$$
\begin{matrix}
\text{At}(\rho^*{\mathcal L})& \stackrel{h_0}{\longrightarrow} &
\rho^*\text{At}({\mathcal L})\\
\,\,\,\Big\downarrow \zeta' && ~\,~\,~\,~\,\Big\downarrow \rho^*\zeta\\
T{\rm Conn}({\mathcal L})& \stackrel{d\rho}{\longrightarrow} & \rho^* T{\mathcal B}_g(r)
\end{matrix}
$$
is commutative, where $\zeta$ and $\zeta'$ are the projections in \eqref{atl} and \eqref{aes2}
respectively, and $d\rho$ is the differential of the projection $\rho$ in
\eqref{e38}; see \cite[(3.9)]{BHS} \cite[Section 3]{BH1} for the construction of $h_0$. The composition of homomorphisms
\begin{equation}\label{fh}
f_0\circ h_0\, :\, \text{At}(\rho^* {\mathcal L})\, \longrightarrow\, {\mathcal O}_{{\mathcal C}({\mathcal L})},
\end{equation}
where $f_0$ is constructed in \eqref{f0}, gives
a holomorphic splitting of the Atiyah exact sequence in \eqref{aes2}. Hence $f_0\circ h_0$
defines a holomorphic connection on $\rho^*{\mathcal L}$; see \cite[Proposition 3.3]{BHS}.

The curvature $\Omega_{\mathcal L}$ of the holomorphic connection $f_0\circ h_0$ on $\rho^*{\mathcal L}$
in \eqref{fh} 
is a closed algebraic $2$--form on ${\rm Conn}({\mathcal L})$. This algebraic $2$--form $\Omega_{\mathcal 
L}$ is symplectic. (See \cite[Section 3]{BH1}, \cite{BHS}.)

Recall the holomorphic $T^*{\mathcal B}_g(r)$--torsor structure of ${\rm Conn}({\mathcal L})$. The above 
symplectic form $\Omega_{\mathcal L}$ on ${\rm Conn}({\mathcal L})$ is compatible with the $T^*{\mathcal 
B}_g(r)$--torsor structure. This means that for any locally defined holomorphic section of the projection $\rho$ 
in \eqref{e38}
$$
{\mathcal B}_g(r)\, \supset\, U\, \stackrel{s}{\longrightarrow}\, {\rm Conn}({\mathcal L})
$$
and any holomorphic $1$--form $\omega\, \in\, H^0(U,\, T^*U)$, we have
\begin{equation}\label{sf1}
s^*\Omega_{\mathcal L} + d\omega \,=\, (s+\omega)^*\Omega_{\mathcal L}\, ;
\end{equation}
note that $y\, \longmapsto\, s(y)+\omega(y)$ is a holomorphic section, over
$U$, of the projection $\rho$ in \eqref{e38}.

\begin{corollary}\label{cor5}\mbox{}
\begin{enumerate}
\item The moduli space ${\mathcal H}_g(r)$ in \eqref{e26} has a canonical algebraic symplectic
structure $\Omega_{{\mathcal H}_g(r)}$.

\item The symplectic form $\Omega_{{\mathcal H}_g(r)}$ on ${\mathcal H}_g(r)$ is compatible with the
$T^*{\mathcal B}_g(r)$--torsor structure of ${\mathcal H}_g(r)$ obtained in Theorem \ref{thm1}.

\item There is a holomorphic line bundle $\mathbf{L}$ on ${\mathcal H}_g(r)$ and a holomorphic
connection $\nabla^{\mathbf{L}}$ on $\mathbf{L}$ such that the curvature of $\nabla^{\mathbf{L}}$
is the symplectic form $\Omega_{{\mathcal H}_g(r)}$.
\end{enumerate}
\end{corollary}

\begin{proof}
Using the isomorphism ${\mathcal H}$ in Theorem \ref{thm3}, the above algebraic symplectic form
$\Omega_{\mathcal L}$ on ${\rm Conn}({\mathcal L})$ (see \eqref{sf1}) produces an algebraic symplectic
form
$$
\Omega_{{\mathcal H}_g(r)}\, :=\, {\mathcal H}^*\Omega_{\mathcal L}
$$
on ${\mathcal H}_g(r)$.

The compatibility condition in the statement (2) says that
for any locally defined holomorphic section of the projection $\gamma$ in \eqref{e26}
$$
{\mathcal B}_g(r)\, \supset\, U\, \stackrel{s}{\longrightarrow}\, {\mathcal H}_g(r)
$$
and any holomorphic $1$--form $\omega\, \in\, H^0(U,\, T^*U)$, the equality
\begin{equation}\label{sfl}
s^*\Omega_{{\mathcal H}_g(r)} + d\omega \,=\, (s+\omega)^*\Omega_{{\mathcal H}_g(r)}
\end{equation}
holds; note that $y\, \longmapsto\, s(y)+\omega(y)$ is a holomorphic section, over $U$, of
the projection $\gamma$ in \eqref{e26}. Now, \eqref{sfl} follows immediately from \eqref{sf1}.

We recall that $\Omega_{\mathcal L}$ is the curvature of the holomorphic connection $f_0\circ h_0$
(see \eqref{fh}) on the holomorphic line bundle $\rho^*{\mathcal L}$. Therefore, $\Omega_{{\mathcal H}_g(r)}$
is the curvature
of the holomorphic connection $$\nabla^{\mathbf{L}}\,:=\, {\mathcal H}^*(f_0\circ h_0)$$ on the
holomorphic line bundle $\mathbf{L}\, :=\, {\mathcal H}^*\rho^*{\mathcal L}$.
\end{proof}

\section{Isomonodromy and symplectic form}

\subsection{Projective structure on a Riemann surface}\label{se9.1}

Let $Y$ be a compact connected Riemann surface. A holomorphic coordinate chart on $Y$ is a pair
of the form $(U,\, \varphi)$, where $U\, \subset\, Y$ is an open subset and $\varphi\,
:\, U\, \longrightarrow\, {\mathbb C}{\mathbb P}^1$ is a holomorphic embedding. A holomorphic
coordinate atlas on $Y$ is a collection of holomorphic coordinate charts $\{(U_i,\, \varphi_i)\}_{i\in J}$
such that $Y\,=\, \bigcup_{i\in J} U_i$. A projective structure on $Y$ is given by a holomorphic
coordinate atlas $\{(U_i,\, \varphi_i)\}_{i\in J}$ satisfying the following condition:
For every $i,\, j\, \in\, J\times J$ with $U_i\cap U_j\, \not=\, \emptyset$,
and every connected component $V_c\, \subset\, U_i\bigcap U_j$, there is a
$\tau^c_{j,i} \, \in\, \text{PSL}(2, {\mathbb C})$ such that the
map $(\varphi_j\circ\varphi^{-1}_i)\big\vert_{\varphi_i(V_c)}$ is the restriction,
to $\varphi_i(V_c)$, of the automorphism of
${\mathbb C}{\mathbb P}^1$ given by $\tau_{j,i}$. Recall that 
$\text{PSL}(2, {\mathbb C})\,=\, \text{Aut}({\mathbb C}{\mathbb P}^1)$.

Two holomorphic coordinate atlases $\{(U_i,\, \varphi_i)\}_{i\in J}$ and $\{(U_i,\, \varphi_i)\}_{i\in J'}$ 
satisfying the above condition are called {\it equivalent} if their union $\{(U_i,\, \varphi_i)\}_{i\in J\cup 
J'}$ also satisfies the above condition. A {\it projective structure} on $Y$ is an equivalence class of 
holomorphic coordinate atlases satisfying the above condition.

Giving a projective structure on $Y$ is equivalent to giving a holomorphic Cartan geometry on
$Y$ for the pair of groups $(\text{PSL}(2,{\mathbb C}),\, B)$, where $B\, \subset\,
\text{PSL}(2,{\mathbb C})$ is the Borel subgroup
$$
B\, :=\, \Big\{
\begin{pmatrix}
a & b\\
c &d
\end{pmatrix}\, \in\, \text{PSL}(2,{\mathbb C})\, \mid\, c\, =\, 0\Big\}\, ;
$$
see \cite{Sh} for Cartan geometry. Let $\widetilde{B}\, \subset\, \text{SL}(2,{\mathbb C})$
be the Borel subgroup that projects to $B$ defined above. 
Giving a holomorphic Cartan geometry on
$Y$ for the pair of groups $(\text{SL}(2,{\mathbb C}),\, \widetilde{B})$ is equivalent to giving
a projective structure on $Y$ together with a theta characteristic on $Y$ \cite{Gu}.

For $g\, \geq\, 2$, let
\begin{equation}\label{eth}
\delta\,\, :\,\, {\mathcal P}_g\, \longrightarrow\, {\mathcal M}^\theta_g
\end{equation}
be the moduli space of triples of the form $(X,\, K^{1/2}_X,\, P)$, where
\begin{itemize}
\item $X$ is a compact connected Riemann surface of genus $g$,

\item $K^{1/2}_X$ is a theta characteristic on $X$, and

\item $P$ is a projective structure on $X$.
\end{itemize}
The map $\delta$ in \eqref{eth} sends any $(X,\, K^{1/2}_X,\, P)$ to $(X,\, K^{1/2}_X)$.

\begin{proposition}\label{prop3}
The moduli space ${\mathcal H}_g(r)$ in \eqref{e26} admits a natural projection
$$
f\,\, :\,\, {\mathcal H}_g(r)\, \longrightarrow\, {\mathcal P}_g\, ,
$$
where ${\mathcal P}_g$ is defined in \eqref{eth}.
\end{proposition}

\begin{proof}
Take any $(X,\, K^{1/2}_X,\, F,\, D)\, \in\, {\mathcal H}_g(r)$. Let
\begin{equation}\label{es}
s\,\, \in\,\,\frac{H^0\left(3\mathbf{\Delta},\, \left((F\otimes K^{3/2}_X)\boxtimes (F^*\otimes K^{3/2}_X)
\otimes{\mathcal O}_{X\times X}(3{\mathbf \Delta})\right)\big\vert_{3\mathbf{\Delta}}\right)}{
H^0(X,\, {\rm ad}(F)\otimes K^{\otimes 2}_X)}
\end{equation}
be the element corresponding to it given by Corollary \ref{cor2}. Restricting $s$ to
$2\mathbf{\Delta}\, \subset\,3\mathbf{\Delta}$, we get a section
\begin{gather}
s_1\, \in\, H^0\left(2\mathbf{\Delta},\, \left((F\otimes K^{3/2}_X)\boxtimes (F^*\otimes K^{3/2}_X)
\otimes{\mathcal O}_{X\times X}(3{\mathbf \Delta})\right)\big\vert_{2\mathbf{\Delta}}\right)\nonumber\\
\,=\, H^0\left(2\mathbf{\Delta},\,\left((F\boxtimes F^*)\otimes ((K^{3/2}_X\boxtimes K^{3/2}_X)\otimes
{\mathcal O}_{X\times X}(3{\mathbf \Delta}))\right)\big\vert_{2\mathbf{\Delta}}\right)\, . \nonumber
\end{gather}
Since $((K^{3/2}_X\boxtimes K^{3/2}_X)\otimes {\mathcal O}_{X\times X}(3{\mathbf \Delta}))\big\vert_{2\mathbf{\Delta}}$ is
canonically trivialized \cite[p.~688, Theorem 2.2]{BR} (this was noted in \eqref{e29a}), there is a unique section
\begin{equation}\label{es1}
s'_1\,\, \in\,\, H^0(2\mathbf{\Delta},\, (F\boxtimes F^*)\big\vert_{2\mathbf{\Delta}})
\end{equation}
such that $s_1\,=\, s'_1\otimes \textbf{t}$, where $\textbf{t}$ is the section of
$((K^{3/2}_X\boxtimes K^{3/2}_X)\otimes {\mathcal O}_{X\times X}(3{\mathbf \Delta}))\big\vert_{2\mathbf{\Delta}}$
that trivializes it.
On the other hand, $s_1\big\vert_{\mathbf{\Delta}}\,=\, \text{Id}_F$ (see Corollary \ref{cor2}),
which implies that $s'_1\big\vert_{\mathbf{\Delta}}\,=\, \text{Id}_F$, because the restriction of
$\textbf{t}$ to $\mathbf{\Delta}$ is the constant function $1$ on $X$ using the identification of
$\mathbf{\Delta}$ with $X$ given by $x\, \longmapsto\, (x,\, x)$. Consequently,
$s'_1$ defines a holomorphic connection on $F$. This holomorphic connection in turn provides an extension of
the section $s'_1$ to a section
$$
s_2\, \in\, H^0\left(3\mathbf{\Delta},\, (F\boxtimes F^*)\big\vert_{3\mathbf{\Delta}}\right)\, ;
$$
we note that using parallel translations, for the integrable connection on $F$ defined by $s'_1$, we get
a holomorphic isomorphism between $q^*_1 F$ and $q^*_2F$ over an analytic neighborhood of
$\mathbf{\Delta}\, \subset\, X\times X$ (recall that $q_i$ is the projection of
$X\times X$ to the $i$-th factor). Now, invoking the isomorphism
$$
s_2\, :\, (q^*_1 F)\big\vert_{3\mathbf{\Delta}}\, \longrightarrow\, (q^*_2 F)\big\vert_{3\mathbf{\Delta}}\, ,
$$
the element $s$ in \eqref{es} becomes
$$
s\,\, \in\,\, \frac{H^0\left(3\mathbf{\Delta},\, \left((F\otimes F^*\otimes K^{3/2}_X)\boxtimes (K^{3/2}_X)
\otimes{\mathcal O}_{X\times X}(3{\mathbf \Delta})\right)\big\vert_{3\mathbf{\Delta}}\right)}{
H^0(X,\, {\rm ad}(F)\otimes K^{\otimes 2}_X)}\, .
$$
Composing with the trace homomorphism $F\otimes F^*\,=\, \text{End}(F)\, \longrightarrow\, {\mathcal O}_X$
defined by $A\, \longmapsto\, \frac{1}{r}\text{trace}(A)$, the above element $s$ gives a section
$$
\phi_2\,\, \in\,\,H^0\left(3\mathbf{\Delta},\, \left((K^{3/2}_X)\boxtimes K^{3/2}_X)
\otimes{\mathcal O}_{X\times X}(3{\mathbf \Delta})\right)\big\vert_{3\mathbf{\Delta}}\right).
$$
This section $\phi_2$ defines a projective structure on $X$ \cite[p.~688, Theorem 2.2]{BR}.
Therefore, we have a map
$$
f\, :\, {\mathcal H}_g(r)\, \longrightarrow\, {\mathcal P}_g
$$
that sends any $(X,\, K^{1/2}_X,\, F,\, D)\, \in\, {\mathcal H}_g(r)$ to $(X,\, K^{1/2}_X,\, \phi_2)
\, \in\,{\mathcal P}_g$, where $\phi_2$ is constructed above from $D$.
\end{proof}

\subsection{Isomonodromy}

For any compact Riemann surface $X$ of genus $g$, with $g\, \geq\, 2$, let
${\mathcal D}_X(r)$ denote the moduli space of pairs of the form $(F,\, D)$, where $F$ is a stable vector
bundle on $X$ of rank $r$ and degree zero, and $D$ is a holomorphic connection on $F$.

\begin{proposition}\label{prop4}
For the map $f$ in Proposition \ref{prop3}, the fiber over any $(X,\, K^{1/2}_X,\, P)\, \in\,{\mathcal P}_g$
is canonically identified with the moduli space ${\mathcal D}_X(r)$.
\end{proposition}

\begin{proof}
Take any $(X,\, K^{1/2}_X,\, F,\, D)\, \in\, {\mathcal H}_g(r)$
in the fiber of $f$ over $(X,\, K^{1/2}_X,\, P)\, \in\,{\mathcal P}_g$. As seen before, $s'_1$ in \eqref{es1}
defines a holomorphic connection on $F$. So we get a map from the fiber of $f$ over
$(X,\, K^{1/2}_X,\, P)$ to ${\mathcal D}_X(r)$ that sends any
$(X,\, K^{1/2}_X,\, F,\, D)$ to $(F,\, s'_1)$.

Conversely, take any $(F,\, D)\, \in\, {\mathcal D}_X(r)$. Then $D$ gives a section
$$
\widehat{s}_D\,\, \in\,\,H^0(3\mathbf{\Delta},\, (F\boxtimes F^*)\big\vert_{3\mathbf{\Delta}})\, .
$$
On the other hand, the projective structure $P$ gives a section
$$
P'\,\, \in\,\, H^0(3\mathbf{\Delta},\, ((K^{3/2}_X\boxtimes K^{3/2}_X)
\otimes{\mathcal O}_{X\times X}(3{\mathbf \Delta}))\big\vert_{3\mathbf{\Delta}})
$$
\cite[p.~688, Theorem 2.2]{BR}. So we have
$$
\widehat{s}_D\otimes P'\,\, \in\,\, H^0\left(3\mathbf{\Delta},\, \left((F\otimes K^{3/2}_X)\boxtimes (F^*\otimes K^{3/2}_X)
\otimes{\mathcal O}_{X\times X}(3{\mathbf \Delta})\right)\big\vert_{3\mathbf{\Delta}}\right)\, .
$$
Consider the image of $\widehat{s}_D\otimes P'$
$$
\widetilde{\widehat{s}_D\otimes P'}\, \in\,
\frac{H^0\left(3\mathbf{\Delta},\, \left((F\otimes K^{3/2}_X)\boxtimes (F^*\otimes K^{3/2}_X)
\otimes{\mathcal O}_{X\times X}(3{\mathbf \Delta})\right)\big\vert_{3\mathbf{\Delta}}\right)}{
H^0(X,\, {\rm ad}(F)\otimes K^{\otimes 2}_X)}
$$
in the quotient space. Now $(X,\, K^{1/2}_X,\, F,\, \widetilde{\widehat{s}_D\otimes P'})$ gives an element of
${\mathcal H}_g(r)$ in the fiber of $f$ over $(X,\, K^{1/2}_X,\, P)\, \in\,{\mathcal P}_g$ (see 
Corollary \ref{cor2}).
\end{proof}

In view of Proposition \ref{prop4}, the isomonodromy condition for
integrable holomorphic connections defines a holomorphic foliation
\begin{equation}\label{imf}
{\mathcal F}\,\, \subset\,\, T{\mathcal H}_g(r)
\end{equation}
which gives the following decomposition:
\begin{equation}\label{imf2}
T{\mathcal H}_g(r)\,=\, {\mathcal F}\oplus \text{kernel}(df)\, ,
\end{equation}
where $df\, :\, T{\mathcal H}_g(r)\, \longrightarrow\, f^*T{\mathcal P}_g$ is the differential of the
projection $f$ in Proposition \ref{prop3}. Consequently, the differential $df$ identifies ${\mathcal F}$ with
$f^*T{\mathcal P}_g$.

We note that ${\mathcal D}_X(r)$ has a natural holomorphic symplectic form \cite{Go}, \cite{AB}. Also,
${\mathcal P}_g$ has a holomorphic symplectic form which is constructed using the monodromy representation
associated to any projective structure \cite{Go}, \cite{AB}, \cite{He}. Therefore, using the decomposition
in \eqref{imf2} we obtain two closed holomorphic $2$-forms on ${\mathcal H}_g(r)$: one is given by the
symplectic form on ${\mathcal D}_X(r)$ and the other is given by the symplectic form on ${\mathcal P}_g$.

We end with the following conjecture:

\begin{conjecture}\label{conj}
The holomorphic symplectic form on ${\mathcal H}_g(r)$ in Corollary \ref{cor5}(1) is a constant linear
combination of the above two holomorphic $2$-forms on ${\mathcal H}_g(r)$.
\end{conjecture}


\end{document}